\documentclass[11pt,twoside]{article}

\usepackage{amsmath, amsthm,amsfonts,amssymb,mathtools,mathdots,scalerel,mathrsfs,stmaryrd,cmll}
\usepackage[utf8]{inputenc}
\usepackage[T1]{fontenc}
\usepackage{mlmodern, tikz-cd, quiver}
\usepackage[many]{tcolorbox}

\usepackage{enumitem}
\usepackage[hidelinks]{hyperref}
\usepackage{tikz, fancyhdr, xparse, xcolor, tocloft}
\usepackage[british]{babel}
\usepackage[a4paper, top=4.5cm, bottom=4.5cm, left=4cm, right=4cm, asymmetric]{geometry}

\usepackage{crossreftools}
\usepackage[textwidth=3cm, textsize=small, colorinlistoftodos]{todonotes}

\usepackage{ahtitle, titlesec}
\usepackage{etoolbox}
\makeatletter
\patchcmd{\ttlh@hang}{\parindent\z@}{\parindent\z@\leavevmode}{}{}
\patchcmd{\ttlh@hang}{\noindent}{}{}{}
\makeatother

\newcommand\eqdef{\coloneqq}
\newcommand\nbd{\nobreakdash-\hspace{0pt}}
\newcommand\idd[1]{\mathrm{id}_{#1}}
\newcommand\bigid[1]{\mathrm{Id}_{#1}}
\newcommand\invrs[1]{#1^{-1}}
\newcommand\after{\circ}

\newcommand\incl{\hookrightarrow}

\newcommand\surj{\to}
\newcommand\restr[2]{{#1}{\raisebox{0pt}{$|_{#2}$}}}
\newcommand\powerset[1]{\mathscr{P}{#1}}

\newcommand\set[1]{\left\{ {#1} \right\}}
\newcommand\order[2]{#2^{(#1)}}

\newcommand{\overbar}[1]{{\mkern 1.5mu\overline{\mkern-1.5mu#1\mkern-1.5mu}\mkern 1.5mu}}

\newcommand\slice[2]{{#1}/{\raisebox{-2pt}{$#2$}}}

\newcommand\cat[1]{\mathbf{#1}}

\newcommand\fun[1]{\mathsf{#1}}
\DeclareMathOperator*{\colim}{colim}

\DeclareMathOperator*{\Ob}{Ob}

\newcommand\Pos{\cat{Pos}}

\newcommand\simplexcat{\Delta}

\newcommand\rdcpx{\cat{RDCpx}}
\newcommand{\atom}{{\scalebox{1.3}{\( \odot \)}}}
\newcommand{\smallatom}{\odot}
\newcommand{\Set}{\cat{Set}}
\newcommand{\sSet}{\cat{sSet}}
\newcommand{\dgmSet}{\atom\Set}

\newcommand{\pt}{\mathbf{1}}

\newcommand\clset[1]{\mathrm{cl}\set{#1}}

\newcommand\maxel[1]{\mathscr{M}\!\mathit{ax}\,#1}
\newcommand\gr[2]{#2_{#1}}

\newcommand\bound[2]{\partial_{#1}^{#2}}
\newcommand\bd[2]{\partial_{#1}^{#2}}

\newcommand\faces[2]{\Delta_{#1}^{#2}}

\newcommand\inter[1]{\mathrm{int}\,#1}

\newcommand\cp[1]{\,{\scriptstyle\#}_{#1}\,}
\newcommand\cpsub[1]{\triangleright_{#1}}
\newcommand\subcp[1]{\prescript{}{#1\!}{\triangleleft}\;}
\newcommand\gencp[1]{\,{\scriptstyle\widehat{\#}}_{#1}\,}

\newcommand\compos[1]{\langle#1\rangle}

\newcommand\celto{\Rightarrow}
\newcommand\eqvto{\overset{\sim}{\Rightarrow}}

\newcommand\submol{\sqsubseteq}

\newcommand\gray{\otimes}
\newcommand\dual[2]{\fun{D}_{#1}{#2}}

\newcommand{\lcyl}[1]{\mathrm{L}_{#1}}
\newcommand{\rcyl}[1]{\mathrm{R}_{#1}}

\newcommand{\hcyl}[1]{\Xi_{#1}}

\newcommand\molec{\mathit{Mol}}
\newcommand\molecin[1]{\slice{\molec}{#1}}
\newcommand\ncat[1]{{#1}\cat{Cat}}
\newcommand\omegacat{\ncat{\omega}}

\DeclareMathOperator{\rd}{Rd}

\DeclareMathOperator{\eqv}{Eqv}
\DeclareMathOperator{\dgn}{Dgn}
\DeclareMathOperator{\nd}{Nd}

\DeclareMathOperator{\cell}{cell}
\DeclareMathOperator{\eqvcell}{eqv}
\DeclareMathOperator{\dgncell}{dgn}
\DeclareMathOperator{\ndcell}{nd}

\newcommand\rev[1]{{#1}^\dag}

\newcommand{\Sd}{\mathrm{Sd}}
\newcommand{\Ex}{\mathrm{Ex}}

\newcommand\arr{\vec{I}}

\renewcommand{\a}{\alpha}

\newcommand{\B}{\mathcal{B}}

\newcommand{\un}{\varepsilon}

\newcommand{\hinv}[1]{\xi_{#1}}

\newcommand\cls[1]{\mathscr{#1}}

\newcommand{\ppair}[2]{[#1, #2]}

\newcommand{\loc}[2]{#1[#2^{-1}]}
\newcommand{\preloc}[2]{#1\set{{#2}^{-1}}}
\newcommand{\selfloc}[1]{\tilde{#1}}
\newcommand{\locarr}{\tilde I}

\newcommand{\rgray}{\fun{\tilde I}}
\newcommand{\mrgray}{\markmol{\fun{\tilde I}}}
\newcommand{\rsigma}{\tilde\sigma}
\newcommand{\riota}{\tilde\iota}
\newcommand{\mriota}{\markmol{\tilde\iota}}
\newcommand{\rGamma}{\tilde\Gamma}
\newcommand{\rpi}{\tilde\pi}
\newcommand{\rnu}{\tilde\nu}
\newcommand{\cyl}{\curvearrowright}

\newcommand{\ppnat}{\mathbin{\Box}}

\newcommand{\mapel}[1]{\iota_{#1}}
\newcommand{\imel}[2]{{#1}_{#2}}

\newcommand{\m}[1]{{\mathsf{m}#1}}
\newcommand{\mdgmSet}{\atom\Set^{\m{}}}
\newcommand{\msSet}{\cat{sSet}^{\m{}}}
\newcommand{\matom}{\markmol{\atom}}
\newcommand{\matomSet}{\matom\Set}
\newcommand{\minmark}[1]{{#1}^{\flat}}
\newcommand{\maxmark}[1]{{#1}^{\sharp}}
\newcommand{\natmark}[1]{{#1}^{\natural}}

\newcommand{\markmol}[1]{{#1}_{\m{}}}

\newcommand{\bdmap}{\partial}
\newcommand{\pgray}{\gray_\mathsf{ps}}
\newcommand{\markarr}{\markmol{\arr}}

\newcommand{\selflocm}[1]{\selfloc{#1}_{\m{}}}

\newcommand{\walkinv}{\Upsilon}
\newcommand{\mwalkinv}{\walkinv_{\m{}}}
\newcommand{\fmwalkinv}{\overbar{\walkinv}_{\m{}}}

\newcommand{\An}{\mathrm{an}}

\newcommand{\Jcomp}{J_{\mathsf{comp}}}
\newcommand{\Jn}[1]{J_{#1}}
\newcommand{\Jhorn}{J_{\mathsf{horn}}}
\newcommand{\Jcoind}{J_{\mathsf{coind}}}
\newcommand{\Jind}{J_{\mathsf{ind}}}
\newcommand{\Jloc}{J_{\mathsf{loc}}}
\newcommand{\Jinv}{J_{\mathsf{inv}}}

\newcommand{\Jat}{J_{\mathsf{at}}}

\newcommand{\C}{\cls{C}}
\makeatletter
\newcommand{\oset}[3][0ex]{%
  \mathrel{\mathop{#3}\limits^{
    \vbox to#1{\kern-1.5\ex@
    \hbox{$\scriptstyle#2$}\vss}}}}
\makeatother
\newcommand\qeq{\oset{?}{=}}

\newcommand\inftyn{\texorpdfstring{$(\infty, n)$}{(∞, n)}}
\newcommand\omegatit{\texorpdfstring{$\omega$}{omega}}

\newtheoremstyle{ittheorem}
  {\topsep}   
  {\topsep}   
  {\itshape}  
  {0pt}       
  {\sffamily \itshape \bfseries} 
  { ---}         
  {5pt plus 1pt minus 1pt} 
  {}          

\newtheoremstyle{itdfn}
  {\topsep}   
  {\topsep}   
  {}  
  {0pt}       
  {\sffamily \itshape \bfseries} 
  {}         
  {5pt plus 1pt minus 1pt} 
  {\thmnumber{#2}{\thmnote{\normalfont\ \ %
  {\sffamily(#3)}.}}}          

\newtheoremstyle{itrmk}
  {0.5\topsep}   
  {0.5\topsep}   
  {\normalfont}  
  {0pt}       
  {\sffamily \itshape} 
  { ---}         
  {5pt plus 1pt minus 1pt} 
  {}          
  
\newtheoremstyle{itexm}
  {0.5\topsep}      
  {0.5\topsep}      
  {\normalfont}     
  {0pt}             
  {\sffamily \itshape \bfseries \color{\mycolor}}        
  {\\}               
  {5pt plus 1pt minus 1pt} 
  {\thmname{#1} \thmnumber{#2}{\thmnote{\normalfont\ \ %
  {\sffamily(#3)}.}}}           
  
\makeatletter
  \renewcommand\@upn{\textit}
\makeatother

\theoremstyle{ittheorem}
\newtheorem{thm}{Theorem}[section]
\newtheorem*{thm*}{Theorem}
\newtheorem{prop}[thm]{Proposition}
\newtheorem*{prop*}{Proposition}
\newtheorem{cor}[thm]{Corollary}
\newtheorem{lem}[thm]{Lemma}
\newtheorem{conj}[thm]{Conjecture}
\newtheorem*{dfn*}{Definition}
\theoremstyle{itdfn}
\newtheorem{dfn}[thm]{}
\theoremstyle{itrmk}
\newtheorem{rmk}[thm]{Remark}
\newtheorem{comm}[thm]{Comment}

\setlist{leftmargin=20pt,itemsep=0pt,parsep=0pt,topsep=1ex}

\titleformat{\section}
 {\large\scshape}{\thesection.}{1em}{}

\titleformat{\subsection}
 {\normalsize\itshape}{\thesubsection.}{1em}{}
\titlespacing*{\subsection}
{0pt}{1.5ex plus 1ex minus .2ex}{1.5ex plus .2ex}

\renewcommand{\cftsecpagefont}{\mdseries}

\makeatletter \renewcommand{\cftsecfillnum}[1]{%
  {\cftsecleader}\nobreak
  \makebox[\@pnumwidth][\cftpnumalign]{\cftsecpagefont \oldstylenums{#1}}\cftsecafterpnum\par
} \makeatother

\newcommand\runtitle{model structures for diagrammatic \( (\infty, n) \)-categories}
\newcommand\runauthor{chanavat and hadzihasanovic}

\title{Model structures for\\diagrammatic \( (\infty, n) \)-categories}

\author{Cl\'emence Chanavat and Amar Hadzihasanovic}

\institution{Tallinn University of Technology}

\begin{document}

\thispagestyle{empty}
\maketitle 

\noindent\makebox[\textwidth][r]{%
	\begin{minipage}[t]{.7\textwidth}
\small \emph{Abstract.}
	Diagrammatic sets admit a notion of internal equivalence in the sense of coinductive weak invertibility, with similar properties to its analogue in strict $\omega$-categories.
	We construct a model structure whose fibrant objects are diagrammatic sets in which every round pasting diagram is equivalent to a single cell---its weak composite---and propose them as a model of $(\infty, \infty)$-categories.
	For each $n < \infty$, we then construct a model structure whose fibrant objects are those $(\infty, \infty)$-categories whose cells in dimension $> n$ are all weakly invertible.
	We show that weak equivalences between fibrant objects are precisely morphisms that are essentially surjective on cells of all dimensions.
	On the way to this result, we also construct model structures for $(\infty, n)$-categories on marked diagrammatic sets, which split into a coinductive and an inductive case when $n = \infty$, and prove that they are Quillen equivalent to the unmarked model structures when $n < \infty$ and in the coinductive case of $n = \infty$.
	Finally, we prove that the $(\infty, 0)$-model structure is Quillen equivalent to the classical model structure on simplicial sets.
	This establishes the first proof of the homotopy hypothesis for a model of $\infty$-groupoids defined as $(\infty, \infty)$-categories whose cells in dimension $> 0$ are all weakly invertible.
\end{minipage}}

\vspace{20pt}

\makeaftertitle

\normalsize

\noindent\makebox[\textwidth][c]{%
\begin{minipage}[t]{.65\textwidth}
\setcounter{tocdepth}{1}
\tableofcontents
\end{minipage}}

\section*{Introduction}

Let \( n \in \mathbb{N} \cup \set{\infty} \). 
We know that many modern mathematical concepts naturally inhabit \( (\infty, n) \)\nbd categories, the same way that most mathematical concepts inhabit \( 1 \)\nbd categories.
While the latter have an unambiguous accepted definition, the notion of \( (\infty, n) \)\nbd category, as currently understood, is best construed as a web of models that, more or less conjecturally, present the same homotopy theory, or better---in a sense yet to be made rigorous---the same \emph{directed} homotopy theory.
This multiplicity of models has been brought about by the failure of strict \( \omega \)\nbd categories \cite{ara2023polygraphs}---seemingly the most obvious candidate, obtained by iterated self-enrichment of the category of small categories---to account for the geometric aspect of higher categories; that is, the failure to satisfy the \emph{homotopy hypothesis} \cite{simpson1998homotopy}, that (roughly) higher groupoids model all classical homotopy types (``spaces''), in such a way that the tower of \( n \)\nbd truncations of a higher groupoid, where all \( n \)\nbd cells connected by an \( (n+1) \)\nbd cell are identified, models the Postnikov tower of the corresponding space.

At present, the web of models has quite a few disconnected components, with two main clusters.
The first may be called the \emph{geometric} or non-algebraic cluster, and includes Segal-type models \cite{rezk2010cartesian, simpson2009homotopy, paoli2019simplicial} as well as ``shaped'' models, based on \emph{marked} or \emph{stratified} presheaves over certain categories of shapes such as simplices and cubes \cite{verity2008weak, campion2020cubical}.
These are characterised by the fact that composition is not an \emph{operation}, but instead there is a space of candidate ``weak composites'' of a diagram whose existence is promised by a fibrancy condition.
Moreover, all these models presuppose some notion of space, so that each \( (\infty, n) \)\nbd category has some distinguished cells which track homotopies in an underlying space.
Thus, in the geometric models, the notion of space has logical priority over the notion of higher category, and the homotopy hypothesis has a somewhat axiomatic character.

It is in the sub-cluster of Segal-type models that a benchmark has first been put forward for comparing models of \( (\infty, n) \)\nbd category, with the work of Barwick and Schommer-Pries \cite{barwick2020unicity}.
Recently, equivalences have been proven between different shaped models \cite{doherty2023equivalence}, as well as between Segal-type and shaped models \cite{ozornova2022quillen, loubaton2023theory}, so it can be said that most models in the cluster currently meet the accepted standard.

The second main cluster may be called the \emph{algebraic} cluster, and includes the Grothendieck--Maltsiniotis \cite{maltsiniotis2010grothendieck} as well as Batanin--Leinster \cite{batanin1998weak, leinster2004higher} models, and more recently ``type-theoretic'' models \cite{mimram2021globular}.
These models attempt to ``fix'' strict \( \omega \)\nbd categories by postulating that their axioms---unitality, associativity, interchange---do not hold strictly up to equality, but instead up to some specified higher-dimensional cells, sometimes called \emph{coherences} or \emph{coherators}; the coherences themselves have (specified) higher coherences, ensuring that the space of possible composites of a diagram is contractible.
Due to their algebraic nature, these models typically adopt an algebraic notion of internal equivalence in the form of \emph{weak invertibility} or \emph{pseudo-invertibility} \cite{cheng2007omega,ozornova2024equivalence}, and define a higher groupoid as a higher category whose cells are all weakly invertible.
Thus, in the algebraic models, higher categories have logical priority, and the homotopy hypothesis is a conjecture with genuine mathematical content, yet, at present, remains unproven, except in low-dimensional cases \cite{joyal2007weak, henry2023homotopy}.

The algebraic models also form a connected cluster, with equivalences having been proven between most of them \cite{ara2010groupoides, bourke2020iterated, benjamin2024invertible}, although notably only at the 1\nbd categorical level: the homotopy theory of these models is significantly less developed, although progress has been made recently on Batanin--Leinster models \cite{fujii2024omega}.

Finally, a number of computer scientists and mathematicians of computational background have taken an interest in higher categories through the lens of \emph{diagrammatic methods} and \emph{higher-dimensional rewriting} \cite{guiraud2019rewriting}.
In this context, an expressive language for cellular presentations of higher algebraic theories, as well as a strong \emph{pasting theorem} guaranteeing that explicit diagrammatic arguments are sound, are fundamental features of a good model of higher categories.
In the absence of these, higher-dimensional rewriting theorists have resorted to the strict model in spite of its limitations \cite{ara2023polygraphs}, or developed natively diagrammatic models \cite{dorn2018associative, corbyn2024homotopy} whose connection to any of the established models is presently unclear.

In this article, which follows the lead given by the second-named author in the unpublished \cite{hadzihasanovic2020diagrammatic} and builds on \cite{chanavat2024diagrammatic, chanavat2024equivalences}, we propose a model of \( (\infty, n) \)\nbd categories which addresses this concern and, we believe, may act as a ``bridge'' between the different clusters, by sharing some of the good features of each flavour of models:
\begin{itemize}
	\item it is, strictly speaking, a shaped model---although, notably, with no marking, or rather, marking can play at most an ancillary role---so it belongs to the geometric family, and satisfies the homotopy hypothesis;
	\item nevertheless, like algebraic models, it supports a notion of internal equivalence in the sense of weak invertibility, so does not presuppose a notion of space: higher groupoids are higher categories whose cells in dimension \( > 0 \) are all weakly invertible; furthermore, weak equivalences between \( (\infty, n) \)\nbd categories are defined, in algebraic style, as functors that are essentially surjective on cells of each dimension;
	\item since it is built upon the combinatorial theory of pasting diagrams from \cite{hadzihasanovic2024combinatorics}, it natively supports a very rich diagrammatic language, which we would claim is actually an \emph{improvement} for the purposes of higher-dimensional rewriting theory over the language of polygraphs, due to the existence of an explicit combinatorial model of diagrams with associated data structures and computational methods \cite{hadzihasanovic2023data, hadzihasanovic2023higher};
	\item in fact, it \emph{feels very much like} the strict model, due to the ability to form pasting diagrams in a way that satisfies associativity and interchange strictly, with ``weakness'' introduced only at the moment of passing to weak composites (which can usually be delayed), and at the level of \emph{units} needed to ``regularise'' diagrams---more on that in a moment.
\end{itemize}
In spirit, the closest predecessor to our model is the opetopic model of Baez and Dolan \cite{baez1998higher}; the presence of many-to-many cell shapes and algebraic units in our model enables us to algebraicise the opetopic notion of universality and turn it into weak invertibility, which overcomes most of its problems.

We give a brief description of the model.
A diagrammatic set is a presheaf on a shape category \( \atom \) whose objects are called \emph{atoms}.
In a specific sense, for each \( m \in \mathbb{N} \), atoms of dimension \( m \) are the widest class of shapes of higher-categorical cells with the property that, for each \( k \leq m \), both the input (source) and output (target) \( k \)\nbd boundary are regular CW models of closed topological \( k \)\nbd balls.
The category \( \atom \) is a very convenient shape category: it is an Eilenberg--Zilber category, a strict test category, and is closed under all sorts of constructions, including suspensions, Gray products, joins, and duals, which can thus be swiftly extended along colimits to all presheaves.

The idea being that atoms are models of directed cells, a diagrammatic set is a model of a directed cell complex.
Now, atoms and their maps are described combinatorially (in terms of \emph{oriented face posets}), and this description extends to the diagrammatic sets that are \emph{regular} cell complexes, just as it happens for CW complexes.
In particular, these include the ``globular pastings of atoms'', which we call \emph{molecules}, amongst which the \emph{round molecules} are the subclass that can appear in the boundary of a higher-dimensional atom, being a molecule as well as a closed topological ball.
This figure depicts, respectively, a non-round molecule, a round molecule, and an atom in dimension 2:
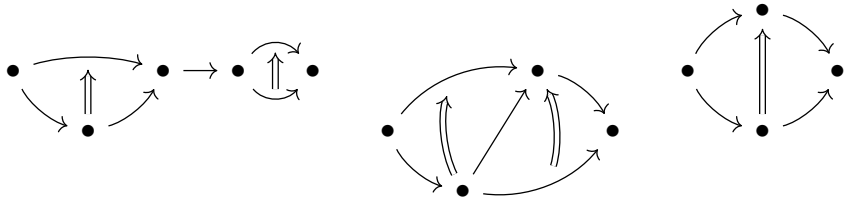
\[\begin{tikzcd}[sep=small]
	&&&&&&&&&& \bullet \\
	\bullet && \bullet & \bullet & \bullet &&& \bullet && \bullet && \bullet \\
	& \bullet &&&& \bullet &&& \bullet && \bullet \\
	&&&&&& \bullet
	\arrow[curve={height=-6pt}, from=1-11, to=2-12]
	\arrow[""{name=0, anchor=center, inner sep=0}, curve={height=-6pt}, from=2-1, to=2-3]
	\arrow[curve={height=6pt}, from=2-1, to=3-2]
	\arrow[from=2-3, to=2-4]
	\arrow[""{name=1, anchor=center, inner sep=0}, curve={height=-12pt}, from=2-4, to=2-5]
	\arrow[""{name=2, anchor=center, inner sep=0}, curve={height=12pt}, from=2-4, to=2-5]
	\arrow[curve={height=-6pt}, from=2-8, to=3-9]
	\arrow[curve={height=-6pt}, from=2-10, to=1-11]
	\arrow[curve={height=6pt}, from=2-10, to=3-11]
	\arrow[curve={height=6pt}, from=3-2, to=2-3]
	\arrow[""{name=3, anchor=center, inner sep=0}, curve={height=-12pt}, from=3-6, to=2-8]
	\arrow[curve={height=6pt}, from=3-6, to=4-7]
	\arrow[Rightarrow, from=3-11, to=1-11]
	\arrow[curve={height=6pt}, from=3-11, to=2-12]
	\arrow[from=4-7, to=2-8]
	\arrow[""{name=4, anchor=center, inner sep=0}, curve={height=12pt}, from=4-7, to=3-9]
	\arrow[shorten <=3pt, shorten >=3pt, Rightarrow, from=2, to=1]
	\arrow[shorten >=4pt, Rightarrow, from=3-2, to=0]
	\arrow[curve={height=6pt}, shorten <=7pt, Rightarrow, from=4, to=2-8]
	\arrow[curve={height=-6pt}, shorten >=7pt, Rightarrow, from=4-7, to=3]
\end{tikzcd}\]
Given a diagrammatic set \( X \), a \emph{pasting diagram in \( X \)} is a morphism whose domain is a molecule.
A pasting diagram is a \emph{round diagram} if the molecule is round, and a \emph{cell} if it is an atom.
Pasting diagrams can be pasted together at their boundaries just like in strict \( \omega \)\nbd categories, and this operation satisfies all the equations of strict \( \omega \)\nbd categories (and more; see Section \ref{sec:folk}), but there is no way, in general, to turn a pasting diagram into a single cell.

On the other hand, there are ``collapsing'' maps of molecules which strictly decrease dimension, such that pulling back a pasting diagram along a collapse produces a \emph{degenerate} pasting diagram.
In particular, there are ``weak unit'' diagrams living above each pasting diagram.
These weak units and degenerate cells play a fundamental role in restoring the expressiveness lost with the roundness constraint on boundaries of cells: indeed, every pasting diagram can be ``padded with units'' until it becomes round.
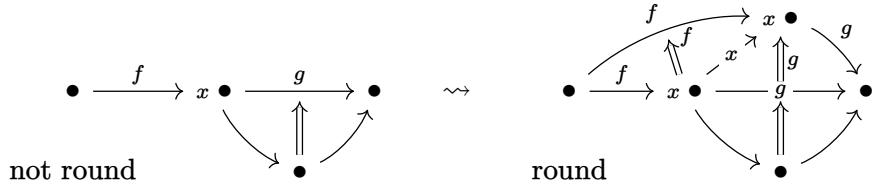
\[\begin{tikzcd}[column sep=small, row sep=scriptsize]
	&&&&&&& {{\scriptstyle x}\;\bullet} \\
	\bullet & {{\scriptstyle x}\;\bullet} && \bullet & \leadsto & \bullet & {{\scriptstyle x}\;\bullet} && \bullet \\
	{\text{not round}} && \bullet &&& {\text{round}} && \bullet
	\arrow["g", curve={height=-6pt}, from=1-8, to=2-9]
	\arrow["f", from=2-1, to=2-2]
	\arrow[""{name=0, anchor=center, inner sep=0}, "g", from=2-2, to=2-4]
	\arrow[curve={height=6pt}, from=2-2, to=3-3]
	\arrow[""{name=1, anchor=center, inner sep=0}, "f", curve={height=-12pt}, from=2-6, to=1-8]
	\arrow["f", from=2-6, to=2-7]
	\arrow["x"{description}, from=2-7, to=1-8]
	\arrow[""{name=2, anchor=center, inner sep=0}, "g"{description}, from=2-7, to=2-9]
	\arrow[curve={height=6pt}, from=2-7, to=3-8]
	\arrow[curve={height=6pt}, from=3-3, to=2-4]
	\arrow[curve={height=6pt}, from=3-8, to=2-9]
	\arrow["f"', shorten >=3pt, Rightarrow, from=2-7, to=1]
	\arrow["g"', shorten <=3pt, Rightarrow, from=2, to=1-8]
	\arrow[shorten >=3pt, Rightarrow, from=3-3, to=0]
	\arrow[shorten >=3pt, Rightarrow, from=3-8, to=2]
\end{tikzcd}\]
The key observation made in \cite{chanavat2024equivalences} is that this structure suffices to instantiate the definition of weak invertibility \emph{at the level of round diagrams} in a diagrammatic set, and this determines a subclass of round diagrams, the \emph{equivalences}, which satisfies analogous properties to its counterpart in strict \( \omega \)\nbd categories.
In particular, it determines an equivalence relation \( \simeq \) on parallel round diagrams of the same dimension.
We finally get to our definition.

\begin{dfn*}
	An \emph{\( (\infty, n) \)\nbd category} is a diagrammatic set \( X \) such that
	\begin{enumerate}
		\item every round diagram in \( X \) is equivalent to a cell, and
		\item every cell of dimension \( > n \) in \( X \) is an equivalence.
	\end{enumerate}
\end{dfn*}
\noindent The first conditions says that every round diagram in dimension \( m \), which may be pasted together from multiple \( m \)\nbd cells, can be reduced to a single \( m \)\nbd cell via a weakly invertible \( (m+1) \)\nbd cell; these play, respectively, the role of a \emph{weak composite} and a \emph{compositor}.
The second condition is self-explanatory, and void when \( n = \infty \).
Now, all morphisms of diagrammatic sets preserve equivalences, so in particular they preserve the property of being a weak composite; a \emph{functor} can thus be defined, simply, as a morphism of presheaves between \( (\infty, n) \)\nbd categories.
We may then import the notion of \emph{\(\omega \)\nbd equivalence} from the theory of strict \( \omega \)\nbd categories.
(``Essentially'' means ``up to \( \simeq \)''.)
\begin{dfn*}
A functor of \( (\infty, n) \)\nbd categories is an \emph{\(\omega \)\nbd equivalence} if
    \begin{enumerate}
	    \item it is essentially surjective on \( 0 \)\nbd cells, and
	    \item for all parallel pairs of round \( m \)\nbd dimensional diagrams in its domain, it is essentially surjective on \( (m+1) \)\nbd cells between their images in its codomain.
    \end{enumerate}
\end{dfn*}
\noindent We state our main theorem.

\begin{thm*}
	There exists a model structure on diagrammatic sets whose
	\begin{itemize}
		\item cofibrations are the monomorphisms,
		\item fibrant objects are precisely the \( (\infty, n) \)\nbd categories,
		\item weak equivalences between fibrants are precisely the \( \omega \)\nbd equivalences.
	\end{itemize}
\end{thm*}
\noindent By definition, an \( (\infty, 0) \)\nbd category in our model is precisely an \( (\infty, \infty) \)\nbd category whose cells in dimension \( > 0 \) are all equivalences.
(No 0\nbd cell can ever be an equivalence.)
We also prove the following result.
\begin{thm*}
	There exist both a left Quillen equivalence and a right Quillen equivalence between the model structure for \( (\infty, 0) \)\nbd categories on diagrammatic sets and the classical model structure on simplicial sets.
\end{thm*}
\noindent To the best of our knowledge, this is the first proof of the homotopy hypothesis for a model in which higher groupoids are defined as higher categories whose cells in dimension \( > 0 \) are all weakly invertible, in direct generalisation of the definition of groupoids as categories whose morphisms are all invertible.


\subsection*{Inductive and coinductive \( (\infty, \infty) \)\nbd categories}

In addition to the \( (\infty, n) \)\nbd model structure on the category of diagrammatic sets, we construct two model structures on the category of \emph{marked diagrammatic sets}, which are to diagrammatic sets what marked simplicial sets are to simplicial sets.
We call these model structures the \emph{inductive} and the \emph{coinductive} \( (\infty, n) \)\nbd model structure.
We then prove that the forgetful functor which forgets the marking is always a Quillen equivalence between the coinductive \( (\infty, n) \)\nbd model structure on marked diagrammatic sets, and the \( (\infty, n) \)\nbd model structure on diagrammatic sets.
Thus, at least in the coinductive case, the marking does not play a fundamental role; this is in contrast to the simplicial and cubical models, where, as far as we know, there is no expectation that unmarked presheaves could support a model of \( (\infty, n) \)\nbd categories.

When \( n < \infty \), we then prove that the inductive and coinductive model structures coincide; but in the case \( n = \infty \), the coinductive model structure is only a left Bousfield localisation of the inductive model structure.
This is in line with the expectations set in \cite{barwick2020unicity}, that there should be a \emph{unique} homotopy theory of \( (\infty, n) \)\nbd categories for \( n < \infty \), but at least two plausible homotopy theories of \( (\infty, \infty) \)\nbd categories.
The two are separated by what we called the \emph{existentialist} and \emph{essentialist} view of equivalences in \cite{chanavat2024equivalences}: whether an equivalence is a cell that \emph{behaves} like an equivalence, or whether it has to be ``distinguished'' as \emph{being} a homotopy in some underlying space.
The rift between the two is typically exemplified by the \( (\infty, \infty) \)\nbd category of cobordisms, where every \( m \)\nbd cobordism is invertible up to an \( (m+1) \)\nbd cobordism, making this an \( (\infty, 0) \)\nbd category in the coinductive sense, but \emph{not} in the inductive sense, where only framed diffeomorphisms and their smooth homotopies are distinguished as ``true'' equivalences.
We note, first, that the inductive notion only makes sense from the geometric perspective where spaces have logical priority, so it is no surprise that unmarked diagrammatic sets would naturally model the coinductive notion; and second, that, as already observed in \cite{loubaton2024inductive}, the natural coinductive notion is not the one predicted in \cite{barwick2020unicity}, which still appears to be in search of an interesting model.

The purpose of introducing marked diagrammatic sets is two-fold.
The first reason is internal: in our model of \( (\infty, n) \)\nbd category, a cell of shape \( U \) is an equivalence if and only if it extends to an explicit localisation \( \selfloc{U} \) of the atom \( U \) at its top-dimensional cell.
This localisation is constructed inductively by attaching a left and a right inverse to \( U \), then attaching two \emph{invertors} one dimension higher, witnessing the invertibility; in turn, these two invertors need to be inverted.
Then, \( \selfloc{U} \) is defined as the colimit of this sequence of cellular extensions.
This provides a concrete cellular model for a ``walking equivalence'' of shape \( U \), in the same ``bi-invertible'' style as the coherent model exhibited for strict \( \omega \)\nbd categories in \cite{hadzihasanovic2024model}, and with the same advantages.
However, when it comes to actual computations, this model forces us to carry around an infinite tower of data, which can be quite cumbersome.
On the other hand, in a fibrant object for the coinductive \( (\infty, n) \) model structure on marked diagrammatic sets, a cell is an equivalence if and only if it is marked; the marking can be seen as a propositional truncation of this infinite tower of data, making computations much more manageable.
As an application, to characterise the fibrant objects of the \( (\infty, n) \)\nbd model structure on diagrammatic sets, we leverage these easier computations in the marked world and import them without much effort to diagrammatic sets via a Quillen equivalence.

The second reason is external: the atom category contains a full and faithful representation of the simplex category, which determines an adjunction between simplicial and diagrammatic sets whose left adjoint is also full and faithful.
This lifts to an adjunction between marked simplicial sets and marked diagrammatic sets.
Having constructed model structures on marked diagrammatic sets thus paves the way towards a comparison with the complicial model \cite{ozornova2020model}, rooting the diagrammatic model firmly into the geometric cluster.


\subsection*{Structure of the article}

\noindent In Section \ref{sec:background}, we review some notions of model category theory and present special cases of the results of \cite{olschok2011left}, which will be used throughout the paper to produce model structures on diagrammatic sets and marked diagrammatic sets.
We do not state the results in their greatest level of generality, but merely the one that applies to all the model structures we wish to construct.

In Section \ref{sec:diagrams}, we start by setting up some terminology about regular directed complexes and diagrams in diagrammatic set, as well as giving some recollection of the theory of equivalences in diagrammatic sets.
We do not give any more background on the theory of atoms, molecules, and regular directed complexes than we have given thus far; instead, we refer the reader to the introductory parts of \cite{chanavat2024equivalences, chanavat2024diagrammatic}, and ultimately, the book \cite{hadzihasanovic2024combinatorics}, in increasing level of detail.
We then move on to marked diagrammatic sets, for which we give two equivalent definitions: one as pairs of a diagrammatic set and a subset of cells containing all degenerate cells, and one as separated presheaves on a site of ``marked atoms''.
Using this equivalent description, we immediately deduce good formal properties for the category of marked diagrammatic sets, such as the fact that it is a locally presentable quasitopos.
We also define a monoidal structure on marked diagrammatic sets, the \emph{pseudo-Gray product}, which marks all pairs of cells both of which have dimension \( > 0 \).
Finally, we define an explicit localisation functor from marked to unmarked diagrammatic sets, freely adding a left and a right weak inverse to each non-degenerate marked cell.
For reasons explained in \cite{hadzihasanovic2024model}, a separate left and right inverse are needed to avoid issues of non-coherence.

In Section \ref{sec:model}, we give our definition of \( (\infty, n) \)\nbd category, and characterise \( (\infty, n) \)\nbd categories in terms of some evident right lifting properties.
Then, we move on to constructing our model structures, starting from marked diagrammatic sets.
We define sets of anodyne extensions, a cellular model, and a cylinder object defined using the pseudo-Gray product with the ``marked arrow''.
We then do the same for unmarked diagrammatic sets, whose cylinder object is the localisation of the marked cylinder object.
For the marked models, we select sets of anodyne extensions that are particularly convenient in that they are \emph{almost} saturated under the conditions that guarantee that they are a generating set of anodyne extensions, with only one easily handled corner case.
For the unmarked models, by contrast, we pick the smallest obvious set.

Section \ref{sec:comparison} contains our main results, comparing and characterising the model structures we have defined.
Section \ref{subsec:ind_coind} concerns the two families of model structures on marked diagrammatic sets.
We characterise the fibrant objects of the coinductive \( (\infty, n) \)\nbd model structure as being precisely the \( (\infty, n) \)\nbd categories with the ``natural'' marking whose marked cells are all and only the equivalences (Theorem \ref{thm:coind_fibrant_characterisation}).
We then show that the coinductive model structure is always a left Bousfield localisation of the inductive (Proposition \ref{prop:ind_bousfield_coind}), and that they actually coincide when \( n < \infty \) (Proposition \ref{prop:n_inductive_equal_n_coinductive}).
Next, in Subsection \ref{subsec:marked_unmarked}, we compare the coinductive \( (\infty, n) \)\nbd model structure on marked diagrammatic sets and the \( (\infty, n) \)\nbd model structure on diagrammatic sets.
We first prove that the adjunction \( \minmark{(-)} \dashv \fun{U} \) between unmarked and marked diagrammatic sets is Quillen, where \( \minmark{(-)} \) only marks degenerate cells and \( \fun{U} \) forgets the marking (Proposition \ref{prop:minmark_is_quillen}).
We use this to characterise the fibrant objects of the \( (\infty, n) \)\nbd model structure as being precisely the \( (\infty, n) \)\nbd categories (Theorem \ref{thm:fibrant_iff_infty_category}), then prove that the adjunction is, in fact, a Quillen equivalence (Theorem \ref{thm:marked_quillen_eq_unmarked}).
Finally, in Section \ref{sub:characterisation_ms_diag_sets}, we characterise the weak equivalences between \( (\infty, n) \)\nbd categories as being precisely the \( \omega \)\nbd equivalences (Theorem 
\ref{thm:w_eq_is_hpty_eq}); this part is essentially an exercise in translation from the folk model structure on strict \( \omega \)\nbd categories \cite{lafont2010folk}, using the formal similarities as well as the foundation laid in \cite{chanavat2024equivalences}.

Section \ref{sec:homotopy} proves the homotopy hypothesis for our model of \( (\infty, 0) \)\nbd categories.
We derive this from our earlier work \cite{chanavat2024diagrammatic}, which established that \( \atom \) is a strict test category, by proving that the \( (\infty, 0) \)\nbd model structure coincides with the ``Cisinski model structure'' obtained abstractly from the general theory of presheaves on test categories (Theorem \ref{thm:cisinski_equal_infty}).

Section \ref{sec:further} serves both as a conclusion and, more importantly, as an overture for future developments.
In it, we propose three conjectures.
The first predicts that the evident adjunction between marked simplicial sets and marked diagrammatic sets will actually be a Quillen equivalence between the \( n \)\nbd complicial model structure and the inductive \( (\infty, n) \)\nbd model structure (Conjecture \ref{conj:complicial_quillen_eq}).
The second predicts a slightly more complicated relation to the folk model structure on strict \( \omega \)\nbd categories.
Indeed, while it possible to define an adjunction between the categories of diagrammatic sets and of strict \( \omega \)\nbd categories, the existence of ``extra equations'' satisfied semantically by the pasting of molecules, and not provable from the axioms of strict \( \omega \)\nbd categories in dimension \( > 3 \), prevents this from being a Quillen adjunction.
Instead, we predict that there is a \emph{span} of right Quillen functors connecting the two, mediated by a category of \emph{stricter} \(\omega\)\nbd categories which satisfy these additional equations (Conjecture \ref{conj:folk_quillen_diag}).
We also suggest that the existence of these extra coherences---which, by our results, are homotopically sound---could have profound implications for the algebraic models, and deserves more scrutiny.
Finally, we predict that the Gray product of diagrammatic sets should be compatible with our model structures, making them monoidal model structures; and that the predicted left adjoint functor from diagrammatic sets to stricter \( \omega \)\nbd categories should be strong monoidal (Conjecture \ref{conj:gray_mon}).

\subsection*{Acknowledgements}

The second-named author was supported by Estonian Research Council grant PSG764.
We thank F\'elix Loubaton for helpful discussions about marked presheaves and the inductive and coinductive models of \( (\infty, \infty) \)\nbd categories.

\section{Background on model structures}  \label{sec:background}

\begin{dfn} [Left and right lifting classes]
    Let \( \C \) be a category and \( S \) be a class of morphisms in \( \C \). 
    We denote by \( l(S) \) and \( r(S) \) the classes of morphisms that have, respectively, the left and right lifting property with respect to all morphisms in \( S \). 
    If \( \C \) has a terminal object \( \pt \), we say that an object \( X \) of \( \C \) has the right lifting property against \( S \) if the unique morphism \( X \to \pt \) does.
\end{dfn}

\begin{dfn} [Cellular model]
    Let \( \C \) be a category. 
    A \emph{cellular model} for \( \C \) is a set \( M \) of monomorphisms such that \( l(r(M)) \) is the class of all monomorphisms of \( \C \).
\end{dfn}

\begin{dfn} [Pushout-product]
    Let \( \C \) be a cocomplete category, \( F, G \colon \C \to \C \) be two endofunctors, and \( \beta \colon F \celto G \) be a natural transformation.
    For all morphisms \( f \colon X \to Y \) in \( \C \), the \emph{pushout-product of \( \beta \) and \( f \)} is the morphism
    \begin{equation*}
        \beta \ppnat f \colon FY \coprod_{FX} GX \to GY
    \end{equation*}
    obtained universally from the naturality square of \( \beta \) at \( f \).
\end{dfn}

\begin{rmk}
	We may always assume that \( \beta \ppnat \idd{X} = \beta_X \).
\end{rmk}

\begin{dfn}[Exact cylinder]
	Let \( \C \) be a locally presentable quasitopos.
	An \emph{exact cylinder on \(\C \)} is an endofunctor \( \fun{I} \) on \( \C \) together with natural transformations
    \[
        (\iota^-, \iota^+) \colon \bigid{\C} \amalg \bigid{\C} \to \fun{I}, 
	\quad \quad
    	\sigma\colon \fun{I} \to \bigid{\C} 
    \] 
    such that
    \begin{enumerate}[align=left]
	   \item[{\crtcrossreflabel{(DH0)}[enum:dh0]}] each component of \( (\iota^-, \iota^+) \) is a monomorphism, and each component of the composite \( \sigma (\iota^-, \iota^+) \) is a codiagonal morphism;
	   \item[{\crtcrossreflabel{(DH1)}[enum:dh1]}] the functor \( \fun{I} \) preserves small colimits and monomorphisms.
    \end{enumerate}
\end{dfn}

\begin{comm}
	Since we assumed \( \C \) to be a locally presentable quasitopos, the notion of exact cylinder and of cartesian cylinder for a weak factorisation system whose left class are monomorphisms \cite[Definition 3.8]{olschok2011left} coincide.
	Indeed, the proof of \cite[Corollary 3.11]{olschok2011left} can be generalized to a locally presentable quasitopos since it has pullback-stable colimits, and by \cite[Proposition 23.8]{wyler1991quasitopos}, whenever we have a diagram
    \[\begin{tikzcd}
	P & B \\
	A & Q \\
	&& X
	\arrow["b", from=1-1, to=1-2]
	\arrow["a", from=1-1, to=2-1]
	\arrow[from=1-2, to=2-2]
	\arrow["y", curve={height=-6pt}, hook, from=1-2, to=3-3]
	\arrow[from=2-1, to=2-2]
	\arrow["x", curve={height=6pt}, hook, from=2-1, to=3-3]
	\arrow["\lrcorner"{anchor=center, pos=0.125, rotate=180}, draw=none, from=2-2, to=1-1]
	\arrow[dashed, from=2-2, to=3-3]
\end{tikzcd}\]
    where \( x \) and \( y \) are monomorphisms, \( P \) is the pullback of \( x \) and \( y \), and \( Q \) is the pushout of \( a \) and \( b \), the dashed map is also a monomorphism.
\end{comm}

\noindent Until the end of the section, we let \( \C \) be a locally presentable quasitopos equipped with 
\begin{enumerate}
	\item an exact cylinder \( (\fun{I}, (\iota^-, \iota^+), \sigma) \), and 
	\item a cellular model \( M \).
\end{enumerate}

\begin{dfn} [Class of anodyne extensions]
	Let \( S \) be a class of morphisms in \( \C \).
	We say that \( S \) is a \emph{class of anodyne extensions} if
	\begin{itemize}
		\item there exists a set \( J \) such that \( S = l(r(J)) \),
		\item for all monomorphisms \( m \) and all \( \a \in \set{+, -} \), \( \iota^\a \ppnat m \) is in \( S \),
		\item \( S \) is closed under the operation \( j \mapsto (\iota^-, \iota^+) \ppnat j \).
	\end{itemize}
	In this case, we call \( J \) a \emph{generating set of anodyne extensions}.
\end{dfn}

\begin{dfn} [Generating anodyne extensions]
	Let \( J \) be a set of monomorphisms in \( \C \).
    	The set of \emph{generating \( J \)-anodyne extensions} is the smallest set \( \An(J) \) which
	\begin{enumerate}
		\item contains \(  J \cup \set{\iota^\a \ppnat m \mid \a \in \set{-, +}, m \in M} \), and 
		\item is closed under the operation \( j \mapsto (\iota^-, \iota^+) \ppnat j \).
	\end{enumerate}
\end{dfn}

\begin{prop} \label{prop:smallest_class_of_anodynes}
	Let \( J \) be a set of monomorphisms in \( \C \).
	Then
	\begin{enumerate}
		\item there exists a smallest class of anodyne extensions containing \( l(r(J)) \),
		\item \( \An(J) \) is a generating set for this class.
	\end{enumerate}
\end{prop}
\begin{proof}
	This is \cite[Proposition 1.3.13]{cisinski2006prefaisceaux}.
\end{proof}

\begin{dfn} [Naive fibration]
    Let \( J \) be a set of monomorphisms and \( f\colon X \to Y \) a morphism in \( \C \).
    We say that \( f \) is a \emph{\( J \)\nbd naive fibration} if it has the right lifting property against \( \An(J) \).
    We say that an object \( X \) is \emph{\( J \)\nbd fibrant} if the unique morphism \( X \to \pt \) is a \( J \)\nbd naive fibration.
\end{dfn}

\begin{dfn} [Homotopy]
    Let \( f, g \colon X \to Y \) be two parallel morphisms in \( \C \).
    A \emph{homotopy} between \( f \) and \( g \) is a morphism \( \beta \colon \fun{I}X \to Y \) such that \( \beta \after \iota^- = f \) and \( \beta \after \iota^+ = g \).
    We say that \( f \) and \( g \) are \emph{homotopic} if there exists a homotopy between them, and let \( \approx \) denote the equivalence relation that this generates.
\end{dfn}

\begin{rmk}
	The relation \( \approx \) is a congruence with respect to composition of morphisms.
\end{rmk}

\begin{dfn} [Homotopy equivalence]
    Let \( f \colon X \to Y \) be a morphism in \( \C \).
    We say that \( f \) is a \emph{homotopy equivalence} if there exists a morphism \( g \colon Y \to X \) such that \( g \after f \) is homotopic to \( \idd{X} \), and \( f \after g \) is homotopic to \( \idd{Y} \).
\end{dfn}

\begin{thm} \label{thm:olschok_theorem}
    Let \( J \) be a set of monomorphisms in \( \C \).
    Then there exists a cofibrantly generated model structure on \( \C \) where
    \begin{itemize}
        \item the cofibrations are the monomorphisms,
	\item an object is fibrant if and only if it is \( J \)\nbd fibrant,
	\item a morphism with fibrant codomain is a fibration if and only if it is a \( J \)\nbd naive fibration,
        \item a morphism \( f \colon X \to Y \) is a weak equivalence if and only if, for all fibrant objects \( W \), the induced function \(  f^* \colon \slice{\C(Y, W)}{\approx} \to \slice{\C(X, W)}{\approx} \) is a bijection of sets,
        \item a morphism between fibrant object is a weak equivalence if and only if it is a homotopy equivalence,
        \item \( (\fun{I}, (\iota^-, \iota^+), \sigma) \) is a cylinder object for the model structure.
    \end{itemize}
\end{thm}
\begin{proof}
    This is \cite[Theorem 3.16, Corollary 3.24, Lemma 3.30]{olschok2011left}
\end{proof}

\begin{rmk} \label{rmk:cisinski_is_olschok}
    The result \cite[Theorem 2.4.19]{cisinski2019higher} is a particular case of the previous theorem where \( \C \) is a category of presheaves.
\end{rmk}

\noindent Let \( J \) be a set of monomorphisms in \( \C \), and suppose \( \C \) is equipped with the model structure of Theorem \ref{thm:olschok_theorem} applied to \( J \).

\begin{dfn} [Local objects and equivalences]
    	Let \( S \) be a class of morphisms in \( \C \).
    	We say that a fibrant object \( W \) in \( \C \) is \emph{\( S \)\nbd local} if, for all morphisms \( f \colon X \to Y \) in \( S \), the function \( f^* \colon \slice{\C(Y, W)}{\approx} \to \slice{\C(X, W)}{\approx} \) is a bijection of sets.
    	We say that a morphism \( f \colon X \to Y \) is an \emph{\( S \)\nbd local equivalence} if, for all \( S \)\nbd local objects \( W \), the function \( f^* \colon \slice{\C(Y, W)}{\approx} \to \slice{\C(X, W)}{\approx} \) is a bijection of sets.
\end{dfn}

\begin{rmk}
    By \cite[Theorem 17.6.7]{hirschhorn2003model} and Theorem \ref{thm:olschok_theorem}, the previous definitions coincide with the usual notions of \( S \)\nbd local object and \( S \)\nbd local equivalence, see \cite[Definition 3.1.4]{hirschhorn2003model}.
\end{rmk}

\begin{dfn} [Left Bousfield localisation]
    	Let \( S \) be a class of morphisms in \( \C \). 
    	The \emph{left Bousfield localisation of \( \C \) at \( S \)} is, if it exists, the unique model structure on \( \C \) whose cofibrations are the monomorphisms, and weak equivalences are the \( S \)\nbd local equivalences.
\end{dfn}

\noindent Every model structure whose cofibrations are the monomorphisms, and weak equivalences contain the weak equivalences of \( \C \), arises as a left Bousfield localisation of \( \C \) at some class of morphisms.

\section{Marked and unmarked diagrammatic sets} \label{sec:diagrams}

\subsection{Complements on regular directed complexes}

We refer to the introduction of \cite{chanavat2024equivalences} for a brief review of the notation relative to atoms, molecules, and regular directed complexes; all the details are in \cite{hadzihasanovic2024combinatorics}.

\begin{dfn}[Atom inclusions]
    Let \( P \) be a regular directed complex. For each \( x \) in \( P \), we denote by \( \mapel x \colon \imel P x \incl P \) the unique inclusion with image \( \clset{x} \) in \( P \).
\end{dfn}

\begin{dfn}[Arrow]
	Let \( \pt \) be the point, that is, the terminal regular directed complex.
	The \emph{arrow} is the atom \( \arr \eqdef \pt \celto \pt \). 
	We denote by \( 0^- < 1 > 0^+ \) the three elements of its underlying poset \( I \), with \( \faces{}{\a} 1 = \set{0^\a} \). 
\end{dfn}

\begin{dfn} [Merger of a round molecule]
    Let \( U \) be a round molecule. 
    The \emph{merger of \( U \)} is the atom \( \compos{U} \eqdef \bd{}{-} U \celto \bd{}{+} U \).
\end{dfn}

\noindent The following definitions are from \cite[Section 1]{chanavat2024equivalences}.

\begin{dfn}[Partial cylinder]
Given a graded poset $P$ and a closed subset $K \subseteq P$, the \emph{partial cylinder on $P$ relative to $K$} is the graded poset $I \times_K P$ obtained as the pushout
\[\begin{tikzcd}
	{I \times K} & K \\
	{I \times P} & {I \times_K  P}
	\arrow[two heads, from=1-1, to=1-2]
	\arrow[hook', from=1-1, to=2-1]
	\arrow["{(-)}", hook', from=1-2, to=2-2]
	\arrow["q", two heads, from=2-1, to=2-2]
	\arrow["\lrcorner"{anchor=center, pos=0.125, rotate=180}, draw=none, from=2-2, to=1-1]
\end{tikzcd}\]
in $\Pos$.
This is equipped with a canonical projection map $\tau_K\colon I \times_K P \surj P$.
\end{dfn}

\begin{dfn}[Partial Gray cylinder]
	Let $U$ be a regular directed complex and $K \subseteq U$ a closed subset.
	The \emph{partial Gray cylinder on $U$ relative to $K$} is the oriented graded poset $\arr \gray_K U$ whose
\begin{itemize}
	\item underlying graded poset is $I \times_K U$, and
	\item orientation is specified, for all $\a \in \set{+, -}$, by
\begin{align*}
	\faces{}{\a}(x) & \eqdef \set{(y) \mid y \in \faces{}{\a}x}, \\
	\faces{}{\a}(i, x) & \eqdef \begin{cases}
		\set{(0^\a, x)} + \set{(1, y) \mid y \in \faces{}{-\a}x \setminus K} &
		\text{if $i = 1$,} \\
		\set{(i, y) \mid y \in \faces{}{\a}x \setminus K} + 
		\set{(y) \mid y \in \faces{}{\a}x \cap K} &
		\text{otherwise}.
	\end{cases}
\end{align*}
\end{itemize}
\end{dfn}

\begin{dfn}[Inverted partial Gray cylinder]
	Let $U$ be a molecule, $n \eqdef \dim{U}$, and $K \subseteq \bd{}{+}U$ a closed subset.
	The \emph{left-inverted partial Gray cylinder on $U$ relative to $K$} is the oriented graded poset $\lcyl{K} U$ whose
\begin{itemize}
	\item underlying graded poset is $I \times_K U$, and
	\item orientation is as in $\arr \gray_K U$, except for all $x \in \gr{n}{U}$ and $\a \in \set{+, -}$
\begin{align*}
	\faces{}{-}(1, x) &\eqdef \set{(0^-, x), (0^+, x)} + \set{(1, y) \mid y \in \faces{}{+}x \setminus K}, \\
	\faces{}{+}(1, x) &\eqdef \set{(1, y) \mid y \in \faces{}{-}x}, \\
	\faces{}{\a}(0^+, x) &\eqdef \set{(0^+, y) \mid y \in \faces{}{-\a}x \setminus K} + 
		\set{(y) \mid y \in \faces{}{-\a}x \cap K}.
\end{align*}
\end{itemize}
	Dually, if $K \subseteq \bd{}{-}U$, the \emph{right-inverted partial Gray cylinder on $U$ relative to $K$} is the oriented graded poset $\rcyl{K}{U}$ whose
\begin{itemize}
	\item underlying graded poset is $I \times_K U$, and
	\item orientation is as in $\arr \gray_K U$, except for all $x \in \gr{n}{U}$ and $\a \in \set{+, -}$
\begin{align*}
	\faces{}{-}(1, x) &\eqdef \set{(1, y) \mid y \in \faces{}{+}x}, \\
	\faces{}{+}(1, x) &\eqdef \set{(0^-, x), (0^+, x)} + \set{(1, y) \mid y \in \faces{}{-}x \setminus K}, \\
	\faces{}{\a}(0^-, x) &\eqdef \set{(0^-, y) \mid y \in \faces{}{-\a}x \setminus K} + 
		\set{(y) \mid y \in \faces{}{-\a}x \cap K}.
\end{align*}
\end{itemize}
\end{dfn}

\begin{rmk} \label{rmk:inverted_cylinder_well_def}
	By \cite[Lemma 1.20, Lemma 1.26]{chanavat2024equivalences}, partial Gray cylinders and inverted partial Gray cylinders respect the classes of molecules, round molecules and atoms.
	Moreover, for all molecules $U$ and closed $K \subseteq U$,
	\begin{itemize}
		\item \( \tau_K\colon \arr \gray_K U \to U \) is a cartesian map of molecules,
		\item if $p\colon U \to V$ is a cartesian map of molecules with $\dim V < \dim U$, then $p \after \tau_K\colon \lcyl{K}{U} \to V$ and $p \after \tau_K\colon \rcyl{K}{U} \to V$ are cartesian maps of molecules.
	\end{itemize}
\end{rmk}

\begin{dfn} [Higher invertor shapes]
    	Let \( U \) be a round molecule.
	The family of \emph{higher invertor shapes on \( U \)} is the family of molecules \( \hcyl s U \) indexed by strings \( s \in \set{L, R}^* \), defined inductively on the length of \( s \) by
	\begin{align*}
		\hcyl{\langle\rangle} U & \eqdef U, \\
		\hcyl{Ls} U &\eqdef \lcyl{\bd{}{+} \hcyl{s} U} (\hcyl{s} U), \\
            	\hcyl{Rs} U &\eqdef \rcyl{\bd{}{-} \hcyl{s} U} (\hcyl{s} U).
        \end{align*}
	These are equipped with cartesian maps \( \tau_s \colon \hcyl{s} U \to U \) of their underlying posets, with the property that for all cartesian maps of molecules \( p\colon U \to V \) such that \( \dim V < \dim U \), the composite \( p \after \tau_s \) is a cartesian map of molecules.
\end{dfn}


\subsection{Diagrams in diagrammatic sets}

\begin{dfn}[Diagram in a diagrammatic set]
	Let $U$ be a regular directed complex and $X$ a diagrammatic set.
	A \emph{diagram of shape $U$ in $X$} is a morphism $u\colon U \to X$.
	A diagram is a \emph{pasting diagram} if $U$ is a molecule, a \emph{round diagram} if $U$ is a round molecule, and a \emph{cell} if $U$ is an atom.
	We write $\dim{u} \eqdef \dim{U}$.
\end{dfn}

\noindent Since isomorphisms of molecules are unique when they exist, we can safely identify pasting diagrams that are isomorphic in the slice of $\dgmSet$ over $X$.
Notice that a cell of shape $U$ in $X$ is the same as an element of $X(U)$.

\begin{dfn}[Sets of cells and round diagrams]
	We let $\rd X$ denote the set of round diagrams in $X$, and $\cell X$ its subset on cells.
	The set $\rd X$ is graded by dimension; given a subset $A \subseteq \rd X$ and $k \in \mathbb{N}$, we let $\gr{k}{A} \eqdef \set{u \in A \mid \dim u = k}$.
	We also let $\gr{>k}{A} \eqdef \bigcup_{n>k} \gr{n}{A}$.
\end{dfn}

\begin{dfn}[Boundaries of pasting diagrams]
	Let $u\colon U \to X$ be a pasting diagram in a diagrammatic set, $k \in \mathbb{N}$, and $\a \in \set{ +, - }$.
	We let $\bd{k}{\a}u \eqdef \restr{u}{\bd{k}{\a}{U}}\colon \bd{k}{\a}U \to X$.
	We may omit the index $k$ when $k = \dim{u} - 1$.
\end{dfn}

\noindent Recall that the class of \emph{submolecule inclusions} is composition-generated by the inclusions \( U, V \incl U \cp{k} V \) of molecules into pastings.
A submolecule inclusion \( \iota\colon V \incl U \) is \emph{rewritable} when \( V \) is round and \( \dim U = \dim V \).

\begin{dfn}[Subdiagram]
	Let $u\colon U \to X$ be a pasting diagram.
	A \emph{subdiagram of $u$} is a pair of a pasting diagram $v\colon V \to X$ and a submolecule inclusion $\iota\colon V \incl U$ such that $v = u \after \iota$.
	A subdiagram is \emph{rewritable} when $\iota$ is a rewritable submolecule inclusion.
	We write $\iota\colon v \submol u$ for the data of a subdiagram of $u$.
\end{dfn}

\begin{dfn}[Pasting of pasting diagrams]
	Let $u\colon U \to X$ and $v\colon V \to X$ be pasting diagrams such that $\bd{k}{+}u = \bd{k}{-}v$.
	We let $u \cp{k} v\colon U \cp{k} V \to X$ be the pasting diagram determined by the universal property of the pasting $U \cp{k} V$.
	More generally, given a subdiagram $\iota\colon \bd{k}{+}u \submol \bd{k}{-}v$, we let $u \cpsub{k,\iota} v\colon U \cpsub{k,\iota} V \to X$ be the pasting diagram determined by the universal property of $U \cpsub{k,\iota} V$ as a pasting of $U$ at a submolecule of $\bd{k}{-} V$.
	Dually, if $\iota\colon \bd{k}{-}v \submol \bd{k}{+}u$, we let $u \subcp{k,\iota} v$ be the universally determined pasting diagram of shape $U \subcp{k,\iota} V$.
\end{dfn}

\noindent We often omit the index $k$ when it is equal to $\min \set{\dim{u}, \dim{v}} - 1$, and omit $\iota$ when it is irrelevant or evident from the context.

\begin{rmk}
	When $\iota$ is an isomorphism, we have $u \cpsub{k,\iota} v = u \subcp{k,\iota} v = u \cp{k} v$.
	Moreover, there are evident subdiagrams $u, v \submol u \cpsub{k,\iota} v$ and $u, v \submol u \subcp{k,\iota} v$ whenever the pastings are defined.
\end{rmk}

\begin{dfn}[Degenerate diagram]
	Let $u\colon U \to X$ be a diagram in a diagrammatic set.
	We say that $u$ is \emph{degenerate} if there exists a diagram $v\colon V \to X$ and a surjective cartesian map of regular directed complexes $p\colon U \surj V$ such that $u = v \after p$ and $\dim{v} < \dim{u}$.
 	We let 
	\begin{align*}
		\dgn X &\eqdef \set{ u \in \rd X \mid \text{$u$ is degenerate} },
		       &\dgncell X &\eqdef \dgn X \cap \cell X, \\
		\nd X &\eqdef \set{ u \in \rd X \mid \text{$u$ is not degenerate} },
			&\ndcell X &\eqdef \nd X \cap \cell X.
	\end{align*}
\end{dfn}

\noindent For each \( n \in \mathbb{N} \), there is a functor \( \dual{n}{} \) on regular directed complexes which reverses the orientation of faces of \( n \)\nbd dimensional elements.

\begin{dfn}[Reverse of a degenerate diagram]
	Let $u\colon U \to X$ be a degenerate diagram in a diagrammatic set, equal to $v \after p$ for some diagram $v\colon V \to X$ and surjective map $p\colon U \surj V$ with $n \eqdef \dim{u} > \dim{v}$.
	The \emph{reverse of $u$} is the degenerate diagram $\rev{u} \eqdef v \after \dual{n}{p}$ of shape $\dual{n}{U}$.
\end{dfn}

\begin{dfn}[Unit]
	Let $u\colon U \to X$ be a pasting diagram.
	The \emph{unit on $u$} is the degenerate pasting diagram $\un u\colon u \celto u$ defined by $u \after \tau_{\bd{}{}U}\colon \arr \otimes_{\bd{}{}U} U \to X$.
\end{dfn}

\begin{dfn}[Equivalence in a diagrammatic set]
	Let $e\colon u \celto v$ be a round diagram in a diagrammatic set $X$.
	We say that $e$ is an \emph{equivalence} if there exists a parallel pair of round diagrams $e^L, e^R\colon v \celto u$ together with equivalences $z\colon e \cp{} e^L \celto \un u$ and $h\colon \un v \celto e^R \cp{} e$.
	In this situation, $e^L$ is called a \emph{left inverse} and $e^R$ is called a \emph{right inverse of $e$}, $z$ is called a \emph{left invertor}, and $h$ is called a \emph{right invertor}.
	We let 
	\[ 
		\eqv X \eqdef \set{ e \in \rd X \mid \text{$e$ is an equivalence}}, 
			\quad 
			\eqvcell X \eqdef \eqv X \cap \cell X.
	\]
	We write $e\colon u \eqvto v$ to indicate that $e\colon u \celto v$ is an equivalence.
\end{dfn}

\begin{comm} \label{comm:coinduction_proof_technique}
	The definition of equivalence is in coinductive style.
	More formally, $\eqv X$ is the greatest fixed point of the operator $\B$ on the power set $\powerset{(\rd X)}$ which sends a set $A \subseteq \rd X$ to the set $\B(A)$ of round diagrams $e\colon u \celto v$ such that there exist round diagrams $e^L, e^R\colon v \celto u$, $z\colon e \cp{} e^L \celto \un u$, and $h\colon \un v \celto e^R \cp{} e$ with $z, h \in A$.
	The corresponding coinductive proof method is: if $A \subseteq \B(A)$, then $A \subseteq \eqv X$.
\end{comm}

\begin{comm}
	This is the \emph{bi-invertibility} definition of equivalence, which will be the most useful for the purposes of this article.
	In \cite[Section 2]{chanavat2024equivalences}, we consider two more definitions---one in terms of ``lax solutions to equations'', and one asking for a single two-sided inverse---both of which are shown to be equivalent in \cite[Theorem 2.28]{chanavat2024equivalences}.
\end{comm}

\begin{dfn}[Equivalent round diagrams]
	Let $u, v$ be a parallel pair of round diagrams in a diagrammatic set $X$.
	We write $u \simeq v$, and say that \emph{$u$ is equivalent to $v$}, if there exists an equivalence $h\colon u \eqvto v$ in $X$.
\end{dfn}

\begin{thm} \label{thm:properties_of_equivalences}
	Let $X$ be a diagrammatic set.
	Then
	\begin{enumerate}
		\item\label{enum:dgn_are_eqv} all degenerate round diagrams in $X$ are equivalences,
		\item\label{enum:congruence} for each \( n \geq 0 \), the relation $\simeq$ is an equivalence relation on \( \gr{n}{\rd X} \), and for \( n > 0 \) it is a congruence with respect to \( - \cp{n-1} - \),
		\item\label{enum:eqv_stable} the set of equivalences in $X$ is closed under $\simeq$,
		\item\label{enum:wkinv_is_equiv} the left and right inverse of an equivalence are both equivalences, and equivalent to each other.
	\end{enumerate}
	Moreover, every morphism $f\colon X \to Y$ sends equivalences to equivalences.
\end{thm}
\begin{proof}
	See \cite[Theorem 2.13, Proposition 2.17, Proposition 2.19, Proposition 2.31]{chanavat2024equivalences}.
\end{proof}


\subsection{Marked diagrammatic sets}

\begin{dfn} [Marked diagrammatic set]
    A \emph{marked diagrammatic set} is a pair of 
    \begin{enumerate}
	    \item a diagrammatic set \( X \), and 
	    \item a set \( A \subseteq \gr{>0}{\cell X} \) of \emph{marked cells} such that \( \dgncell X \subseteq A \).
    \end{enumerate}
    A \emph{morphism} \( f \colon (X, A) \to (Y, B) \) of marked diagrammatic sets is a morphism \( f \colon X \to Y \) of diagrammatic sets such that \( f(A) \subseteq B \).
    This determines a category \( \mdgmSet \) of marked diagrammatic sets and morphisms.
\end{dfn}


\begin{dfn} [Flat, sharp, and natural marking]
	There is an evident forgetful functor \( \fun{U} \colon \mdgmSet \to \dgmSet \) that returns the underlying diagrammatic set.
	This functor has both a right and a left adjoint
    \begin{equation*}
        \minmark{(-)} \dashv \fun{U} \dashv \maxmark{(-)}
    \end{equation*}
    defined by \( \minmark{X} \eqdef (X, \dgncell X) \) and \( \maxmark{X} \eqdef (X, \gr{>0}{\cell X}) \).
    By Theorem \ref{thm:properties_of_equivalences}, there is also a functor \( \natmark{(-)}\colon \dgmSet \to \mdgmSet \) defined by \( \natmark{X} \eqdef (X, \eqvcell X) \).
\end{dfn}

\begin{rmk}
	The forgetful functor \( \fun{U} \) is left inverse to \( \minmark{(-)} \), \( \natmark{(-)} \), and \( \maxmark{(-)} \).
\end{rmk}

\noindent We take a small detour and give an equivalent description of the category of marked diagrammatic sets, which will imply it is a locally presentable quasitopos, so that the results of Section \ref{sec:background} apply.
Recall that maps in \( \atom \) factor uniquely as dimension-non-increasing \emph{collapses} followed by dimension-non-decreasing \emph{inclusions}.
The following generalises \cite[Notation 1.1]{ozornova2020model} from marked simplicial sets to marked diagrammatic sets.

\begin{dfn} [Category of marked atoms]
    Let \( \matom \) be the category defined as follows. 
    The set of objects is given by
    \begin{equation*}
	    \Ob(\matom) \eqdef \Ob\atom \cup \set{\markmol{U} \mid U \in \Ob\atom}.
    \end{equation*} 
    The morphisms are generated by the families of morphisms
    \begin{enumerate}
        \item \( f\colon U \to V \), indexed by cartesian maps \( f\colon U \to V \) in \( \atom \),
	\item \( t_U\colon U \to \markmol{U} \), indexed by atoms \( U \) of dimension $> 0$, and
	\item \( \markmol{p} \colon \markmol{U} \to V \), indexed by non-identity collapses \( p\colon U \surj V \) in \( \atom \),
\end{enumerate}
    subject to the following relations:
    \begin{itemize}
        \item for each composable pair \( f, g \) in \( \atom \), \( g \after f \) is the same as in \( \atom \);
	\item for each non-identity collapse \( p \colon U \surj V \), \( p = \markmol{p} \after t_U \);
	\item for each composable pair \( p, q \) of non-identity collapses, \( q \after \markmol{p} = \markmol{(q \after p)} \).
    \end{itemize}
    We denote by \( \matomSet \) the category of presheaves over \( \matom \).
\end{dfn}

\begin{prop} \label{prop:marked_dset_quasitopos}
	There exists a full and faithful functor
	\[
		i\colon \mdgmSet \incl \matomSet
	\]
	whose essential image consists of the separated presheaves for the coverage on \( \matom \) whose only non-identity covers are of the form \( \{ t_U \colon U \to \markmol{U} \} \).
\end{prop}
\begin{proof}
    We define \( i \) as follows: given a marked diagrammatic set \( (X, A) \), we let \( i(X, A) \) be the presheaf sending \( U \) to \( X(U) \) and \( \markmol{U} \) to \( A \cap X(U) \).
    The action of maps \( f\colon U \to V \) in \( \atom \) is the same as on \( X \), the morphisms \( t_U\colon U \to \markmol{U} \) act as the inclusions \( A \cap X(U) \incl X(U) \), and the morphisms \( \markmol{p} \colon \markmol{U} \to U \) act as the corestrictions of \( X(p)\colon X(V) \to X(U) \) to \( A \cap X(U) \), which is well-defined because $A$ contains all degenerate cells.
    Given a morphism \( f \colon (X, A) \to (Y, B) \), we let \( i(f) \colon i(X, A) \to i(Y, B) \) be the evident morphism of presheaves.
    The functor is evidently faithful, and its essential image consists of the presheaves such that \( X(t_U) \) is injective for each atom \( U \), which are precisely the separated presheaves for the given coverage.
    Finally, for any morphism \( f\colon X \to Y \) of presheaves in \( \matomSet \) such that \( Y \) is separated, the naturality square at \( t_U \) implies that \( f\colon X({\markmol{U}}) \to Y({\markmol{U}}) \) is uniquely determined by \( f\colon X(U) \to Y(U) \).
    It follows that $i$ is also full.
\end{proof}

\begin{cor} \label{cor:marked_dset_limits_colimits}
    The category \( \mdgmSet \) is a locally presentable quasitopos; in particular, it is cartesian closed, complete, and cocomplete.
    Given a diagram \( \fun{F} \colon \cls{J} \to \mdgmSet \),
    \begin{enumerate}
	    \item given a limit cone \( (\pi_j \colon X \to \fun{UF}j)_{j \in \Ob\cls{J}} \) over $\fun{UF}$ in \( \dgmSet \) and
		    \[
	A^\pi \eqdef \set{u \in \cell X \mid \text{for all $j \in \Ob \cls{J}$, $\pi_j(u)$ is marked in $\fun{F}j$}},
\]
	the cone \( (\pi_j\colon (X, A^\pi) \to \fun{F}j)_{j \in \Ob\cls{J}} \) is a limit cone over $\fun{F}$ in $\mdgmSet$;
\item given a colimit cone \( (\iota_j \colon \fun{UF}j \to X)_{j \in \Ob\cls{J}} \) under $\fun{UF}$ in \( \dgmSet \) and
	\[
		A_\iota \eqdef \bigcup_{j \in \Ob\cls{J}} \set{ \iota_j(u) \mid \text{$u$ is marked in $\fun{F}j$} },
	\]
	the cone \( (\iota_j\colon \fun{F}j \to (X, A_\iota))_{j \in \Ob\cls{J}} \) is a colimit cone under $\fun{F}$ in $\mdgmSet$.
\end{enumerate}
\end{cor}

\begin{dfn} [Entire and regular monomorphisms]
    Let \( i \colon (X, A) \incl (Y, B) \) be a monomorphism of marked diagrammatic sets. 
    We say that \( i \) is
    \begin{itemize}
        \item \emph{entire} if it is the identity on the underlying diagrammatic set,
        \item \emph{regular} if \( B = i(A) \cup \dgncell Y \).
    \end{itemize}
\end{dfn}

\begin{rmk} \label{rmk:facto_mono_entire_regular}
    Any monomorphism \( i \colon (X, A) \incl (Y, B) \) of marked diagrammatic sets factors uniquely as a regular monomorphism followed by an entire monomorphism:
    \[
        \begin{tikzcd}
            {(X, A)} & {(Y, i(A) \cup \dgncell Y)} & {(Y, B).}
            \arrow[hook, from=1-1, to=1-2]
            \arrow[hook, from=1-2, to=1-3]
        \end{tikzcd}
\]
\end{rmk}

\begin{dfn} [Induced marking]
    Let \( X \) be a diagrammatic set, \( (Y, B) \) be a marked diagrammatic set, and \( i \colon X \incl Y \) be a monomorphism of diagrammatic sets.
    The \emph{marking induced by \( i \)} is the marking \( (X, \invrs{i}B) \) on \( X \).
    This determines a monomorphism \( (X, \invrs{i}B) \incl (Y, B) \) of marked diagrammatic sets.
\end{dfn}

\noindent We refer to \cite[Section 1.2]{chanavat2024diagrammatic} for the definition of the Gray product of diagrammatic sets.
The following definition is adapted from \cite[Construction 2.17]{loubaton2024inductive} to the context of marked diagrammatic sets.

\begin{dfn} [Pseudo-Gray product]
    Let \( (X, A) \) and \( (Y, B) \) be marked diagrammatic sets.
    The \emph{pseudo-Gray product of \( (X, A) \) and \( (Y, B) \)} is the marked diagrammatic set \( (X, A) \pgray (Y, B) \eqdef (X \gray Y, A \pgray B) \), where    
    \[
        A \pgray B \eqdef 
	\dgncell (X \gray Y) \cup
	\left(\gr{0}{\cell X} \gray B\right) \cup 
	\left(\gr{> 0}{\cell X} \gray \gr{> 0}{\cell Y}\right) \cup 
	\left(A \gray \gr{0}{\cell Y}\right).
\]
This construction extends to a monoidal structure on \( \mdgmSet \), whose monoidal unit is the terminal object \( \pt \), such that \( \fun{U}\colon (\mdgmSet, \pgray, \pt) \to (\dgmSet, \gray, \pt) \) is a strict monoidal functor.
\end{dfn}

\begin{lem} \label{lem:marked_gray_product_preserves_mono}
    Let \( m \) and \( m' \) be monomorphisms in \( \mdgmSet \). 
    Then \( m \pgray m' \) is a monomorphism.
\end{lem}
\begin{proof}
    The functor \( \fun{U} \) creates monomorphisms, and since it is monoidal, we conclude by \cite[Lemma 3.5]{chanavat2024diagrammatic}.  
\end{proof}

\begin{lem} \label{lem:marked_gray_preserve_colimits}
    The pseudo-Gray product preserves colimits in both variables.
\end{lem}
\begin{proof}
    The Gray product of diagrammatic sets is part of a biclosed monoidal structure, so it preserves colimits in both variables.
    We conclude by Corollary \ref{cor:marked_dset_limits_colimits} and by inspection of the definition.
\end{proof}

\begin{dfn} [Marked regular directed complex]
	A \emph{marked regular directed complex} $(P, A)$ is a regular directed complex \( P \) together with a set \( A \subseteq \gr{>0}{P} \) of marked elements.
	We will identify a marked regular directed complex \( (P, A) \) with the marked diagrammatic set whose set of marked cells is
\[
	    \dgncell P \cup \set{ \mapel{x}\colon \imel{P}{x} \to P \mid x \in A }.
\]
	If \( P \) is a molecule or an atom, we speak of a \emph{marked molecule} or \emph{marked atom}.
	Given a molecule \( U \), we let \( \markmol{U} \) denote the marked molecule \( (U, {\maxel{U}}) \). 
\end{dfn}

\begin{rmk}
	The representable presheaf \( \markmol{U} \) on \( \matom \) coincides with the image of the marked atom \( \markmol{U} \) through \( i\colon \mdgmSet \to \matomSet \).
\end{rmk}



\subsection{Localisation of diagrammatic sets}

\begin{dfn} [Cellular extension]
	Let \( X \) be a diagrammatic set.
	A \emph{cellular extension of \( X \)} is a diagrammatic set \( X_{S} \) together with a pushout diagram
	\[
	\begin{tikzcd}
		{\coprod_{e \in S} \bd{}{}U_e} &&& {\coprod_{u \in S} U_e} \\
		X &&& {X_S}
		\arrow["{(\bound{}{}e)_{e \in S}}", from=1-1, to=2-1]
		\arrow["{(e)_{e \in S}}", from=1-4, to=2-4]
	\arrow[hook, from=2-1, to=2-4]
	\arrow["{\coprod_{e \in S}\bdmap_{U_e}}", hook, from=1-1, to=1-4]
	\arrow["\lrcorner"{anchor=center, pos=0.125, rotate=180}, draw=none, from=2-4, to=1-1]
\end{tikzcd}\]
in \( \dgmSet \) such that $U_e$ is an atom and $\bdmap_{U_e}$ is the inclusion of its boundary for each \( e \in S \).
	Each \( e \in S \) determines a cell $e\colon e^- \celto e^+$ in $X_{S}$.
	In turn, the pushout is determined by the set of pairs of round diagrams \( \set{(e^-, e^+)}_{e \in S} \) in $X$.
	We say that \( X_{S} \) is the result of \emph{attaching the cells \(\set{ e\colon e^- \celto e^+ }_{e \in S} \) to \( X \)}.
\end{dfn}

\begin{dfn} [Localisation of a diagrammatic set]
	Let \( (X, A) \) be a marked diagrammatic set.
    	We define \( \preloc{X}{A} \) to be the diagrammatic set obtained from \( X \) in the following two steps: for each cell $a\colon u \celto v$ in $A \cap \ndcell X$,
    \begin{enumerate}
        \item attach cells \( a^L \colon v \celto u \) and \( a^R \colon v \celto u \), then
        \item attach cells \( \hinv{L}(a) \colon a \cp{} a^L \celto \un(u) \) and \( \hinv{R}(a) \colon \un(v) \celto a^R \cp{} a \).
    \end{enumerate}
    Let \( \order{0}{X} \eqdef X \) and \( \order{0}{A} \eqdef A \).
    Inductively, for each \( n > 0 \), we let
    \begin{equation*}
    	\order{n}{X} \eqdef \preloc{\order{n-1}{X}}{\order{n - 1}{A}}, 
	\quad\quad 
	\order{n}{A} \eqdef \set{\hinv{L}(a), \hinv{R}(a) \mid a \in \order{n-1}{A}}.
    \end{equation*}
    We have a sequence \( (\order{n}{X} \incl \order{n+1}{X})_{n \geq 0} \) of inclusions of diagrammatic sets. 
    The \emph{localisation of \( X \) at \( A \)} is the colimit \( \loc{X}{A} \) of this sequence.

    For each cell \( a \) of shape \( U \) in \( A \), we define a family of cells \( \hinv{s}a \) of shape \( \hcyl{s}U \) in \( \loc{X}{A} \), indexed by strings $s \in \set{L, R}^*$, by
	\[
	\begin{array}{lc}
		\;\hinv{\langle\rangle}a  \eqdef a & \\
		\hinv{Ls}a \eqdef
		\begin{cases}
			\hinv{L}(\hinv{s}a)
				& \text{if \( a \in \ndcell X \)}, \\
			a \after \tau_s
				& \text{if \( a \in \dgncell X \)},
		\end{cases} &
		\hinv{Rs}a \eqdef 
		\begin{cases}
			\hinv{R}(\hinv{s}a)
				& \text{if \( a \in \ndcell X \)}, \\
			a \after \tau_s
				& \text{if \( a \in \dgncell X \)},
		\end{cases}
	\end{array}
	\]
	where \( \tau_s \) is the projection \( \hcyl{s}U \to U \).
	We also let, for each \( s \in \set{L, R}^* \),
	\[
	\hinv{s}^La \eqdef 
	\begin{cases}
		(\hinv{s}a)^L
			& \text{if \( a \in \ndcell X \)}, \\
		\rev{(a \after \tau_s)}
			& \text{if \( a \in \dgncell X \)},
	\end{cases} \quad 
	\hinv{s}^Ra \eqdef 
	\begin{cases}
		(\hinv{s}a)^R
			& \text{if \( a \in \ndcell X \)}, \\
		\rev{(a \after \tau_s)}
			& \text{if \( a \in \dgncell X \)}.
	\end{cases}
	\]
\end{dfn} 

\noindent We will identify diagrams in \( X \) with their image through \( X \incl \loc{X}{A} \).
By construction, every cell in \( A \) becomes an equivalence in \( \loc{X}{A} \).

\begin{prop} \label{prop:functorial_localisation}
	Let \( f\colon (X, A) \to (Y, B) \) be a morphism of marked diagrammatic sets.
	Then there is a unique morphism \( \fun{Loc}f\colon \loc{X}{A} \to \loc{Y}{B} \) such that
	\begin{enumerate}
		\item \( \fun{Loc}f \) is equal to \( f \) on \( X \incl \loc{X}{A} \),
		\item for each \( a \in A \) and \( s \in \set{L, R}^* \), \( \fun{Loc}f\colon \hinv{s}a \mapsto \hinv{s}f(a) \), \( \hinv{s}^La \mapsto \hinv{s}^Lf(a) \), and \( \hinv{s}^Ra \mapsto \hinv{s}^Rf(a) \).
	\end{enumerate}
	This assignment determines a functor \( \fun{Loc}\colon \mdgmSet \to \dgmSet \), which preserves monomorphisms and colimits.
\end{prop}
\begin{proof}
	Each cell in \( \ndcell \loc{X}{A} \) is either in the image of \( X \incl \loc{X}{A} \), or it is of the form \( \hinv{s}{a} \), \( \hinv{s}^La \), or \( \hinv{s}^Ra \) for some \( a \in A \cap \ndcell X \) and \( s \in \set{ L, R }^* \) of positive length.
	Thus the assignment specifies \( \fun{Loc} f \) uniquely, and an easy induction shows that it is well-defined as a morphism of diagrammatic sets.
	Functoriality and preservation of monomorphisms are straightforward checks.
	
	Finally, let \( \fun{F}\colon \cls{J} \to \mdgmSet \) be a diagram and let \( (\iota_j\colon \fun{F}j \to (X, A_\iota))_{j \in \Ob\cls{J}} \) be a colimit cone under \( \fun{F} \) as described in Corollary \ref{cor:marked_dset_limits_colimits}; this is preserved by \( \fun{U} \).	
	Let \( (\iota'_j\colon \fun{LocF}j \to Y )_{j \in \Ob\cls{J}} \) be a colimit cone under \( \fun{LocF} \) in \( \dgmSet \).
	Then each non-degenerate cell in \( \loc{X}{A_\iota} \) is either equal to \( \iota_j(u) \) for a cell \( u \) of \( \fun{F}j \), or to \( \hinv{s}\iota_j(u) \), \( \hinv{s}^L\iota_j(u) \), or \( \hinv{s}^R\iota_j(u) \) for a marked cell \( u \) of \( \fun{F}j \), for some \( j \in \Ob\cls{J} \).
	Thus we can define a morphism \( \loc{X}{A_\iota} \to Y \) by \( u \mapsto \iota'_j(u) \) in the first case, and \( \hinv{s}\iota_j(u) \mapsto \iota'_j(\hinv{s}u) \), \( \hinv{s}^L\iota_j(u) \mapsto \iota'_j(\hinv{s}^Lu) \), and  \( \hinv{s}^R\iota_j(u) \mapsto \iota'_j(\hinv{s}^Ru) \) in the other cases.
	It is straightforward to check that this is an isomorphism.
\end{proof}

\begin{rmk}
	For all diagrammatic sets \( X \), we have \( \fun{U}(\minmark{X}) = X = \fun{Loc}(\minmark{X}) \).
\end{rmk}

\begin{dfn} [Walking equivalence]
	Let \( U \) be an atom of dimension \( > 0 \).
	The \emph{walking equivalence of shape \( U \)} is the diagrammatic set \( \selfloc U \eqdef \fun{Loc}(\markmol{U}) \).
	If \( U \) is equal to \( V \celto W \), we also write \( V \eqvto W \) for \(\selfloc{U}\).
\end{dfn}

\noindent As a particular case, we have the following object, which is the analogue of the ``coherent walking \( \omega \)\nbd equivalence'' of \cite{hadzihasanovic2024model}.

\begin{dfn} [Reversible arrow]
	The \emph{reversible arrow} is \( \locarr \eqdef \pt \eqvto \pt \).
\end{dfn}

\section{Model structures} \label{sec:model}

\subsection{Diagrammatic \inftyn-categories}

\begin{dfn}[Diagrammatic $(\infty, \infty)$-category]
	A \emph{diagrammatic \( (\infty, \infty) \)-category} is a diagrammatic set \( X \) with the following property: for all round diagrams \( u \) of shape \( U \) in \( X \), there exists a cell \( \compos{u} \) of shape \( \compos{U} \) such that \( u \simeq \compos{u} \).
	In this case, we call \( \compos{u} \) a \emph{weak composite of \( u \)}, and we call an equivalence \( c\colon u \eqvto \compos{u} \) a \emph{compositor for \( u \)}.
\end{dfn}

\begin{rmk}
	It is immediate from the definition that weak composites are unique up to equivalence.
\end{rmk}

\begin{dfn}[Functor of diagrammatic $(\infty, \infty)$-categories]
	A morphism \( f\colon X \to Y \) of diagrammatic sets is called a \emph{functor} when \( X \) and \( Y \) are diagrammatic \( (\infty, \infty) \)-categories.
\end{dfn}

\begin{dfn}[Diagrammatic $(\infty, n)$-category]
	Let \( n \in \mathbb{N} \).
	A \emph{diagrammatic \( (\infty, n) \)\nbd category} is a diagrammatic \( (\infty, \infty) \)-category with the property that every cell of dimension \( > n \) is an equivalence.
\end{dfn}

\noindent In what follows, we will speak simply of \( (\infty, \infty) \)\nbd categories and of \( (\infty, n) \)\nbd categories; we will also often let \( n \) range over \( \mathbb{N} \cup \set{\infty} \).
We will also say that an \( (\infty, \infty) \)\nbd category \emph{has weak composites}.

\begin{rmk}
	By \cite[Proposition 2.15]{chanavat2024equivalences}, in an \( (\infty, n) \)-category, not only every cell but also every round diagram of dimension \( > n \) is an equivalence.
\end{rmk}

\noindent Next, we give a characterisation of \( (\infty, \infty) \)-categories and of \( (\infty, n) \)-categories in terms of lifting properties.

\begin{lem} \label{lem:equivalences_exhibited_by_cells}
	Let \( X \) be an \( (\infty, \infty) \)-category, let \( u, v \) be parallel round diagrams in \( X \), and suppose \( u \simeq v \).
	Then there exists a cell \( e\colon u \eqvto v \) exhibiting the equivalence.
\end{lem}
\begin{proof}
	By definition, there exists an equivalence \( e'\colon u \eqvto v \), which is a priori only a round diagram.
	Because \( X \) has weak composites, we can take a weak composite \( e \eqdef \compos{e'} \simeq e' \).
	By Theorem \ref{thm:properties_of_equivalences}.\ref{enum:eqv_stable}, \( e \) is still an equivalence.
\end{proof}

\begin{prop} \label{prop:universal_property_localisation}
	Let \( (X, A) \) be a marked diagrammatic set, let \( Y \) be an \( (\infty, \infty) \)-category, and let \( f\colon X \to Y \) be a morphism.
    	The following are equivalent:
    \begin{enumerate}[label=(\alph*)]
        \item for all \( a \in A \), the cell \( f(x) \) is an equivalence in \( Y \);
        \item there exists a map \( \tilde f \colon \loc{X}{A} \to Y \) such that the triangle
\[\begin{tikzcd}
	X & Y \\
	{\loc{X}{A}}
	\arrow["f", from=1-1, to=1-2]
	\arrow[hook', from=1-1, to=2-1]
	\arrow["{\tilde{f}}"', from=2-1, to=1-2]
\end{tikzcd}\]        
	commutes.
    \end{enumerate}
\end{prop}
\begin{proof}
	By definition, \( \loc{X}{A} \) is the colimit of a sequence \( \order{n}{X} \incl \order{n+1}{X} \), starting with \( \order{0}{X} \equiv X \).
	We define \( \tilde{f} \) by successive extensions from \( \order{n}{X} \) to \( \order{n+1}{X} \), starting with \( \order{0}{f} \eqdef f \), with the property that \( \order{n}{f}(a) \) is an equivalence for all \(a \in \order{n}{A} \).
	By assumption, \( \order{0}{f}(a) \) is an equivalence for all \(a \in \order{0}{A} \equiv A \cap \ndcell X \).
	Let \( n > 0 \); by the inductive hypothesis, \( g \eqdef \order{n-1}{f} \) sends all \( a \in \order{n-1}{A} \) to equivalences.
	Let \( a \in \order{n-1}{A} \).
	Then \( b \eqdef g(a)\colon u \celto v \) has a left and a right inverse \( b^L\colon v \celto u \), \( b^R\colon v \celto u \), such that \( b \cp{} b^L \simeq \un{u} \) and \( \un{v} \simeq b^R \cp{} b \).
	By assumption, there are weak composites \( b^L \simeq \compos{b^L} \) and \( b^R \simeq \compos{b^R} \), so by Theorem \ref{thm:properties_of_equivalences}.\ref{enum:congruence} we have \( b \cp{} \compos{b^L} \simeq \un{u} \) and \( \un{v} \simeq \compos{b^R} \cp{} b \).
	By Lemma \ref{lem:equivalences_exhibited_by_cells}, these equivalences are exhibited by cells \( z \) and \( h \), respectively. 
	Letting
	\[
		a^L \mapsto \compos{b^L}, \quad a^R \mapsto \compos{b^R}, \quad \hinv{L}(a) \mapsto z, \quad \hinv{R}(a) \mapsto h
	\]
	determines an extension \( \order{n}{f} \) of \( g \) to \( \order{n}{X} \), which by construction sends cells in \( \order{n}{A} \) to equivalences in \( Y \).
	Passing to the colimit, we obtain a morphism $\tilde{f}$ which extends $f$, so that the triangle commutes.
	For the other direction, by construction each \( a \in A \) is an equivalence in \( \loc{X}{A} \), so by Theorem \ref{thm:properties_of_equivalences} \( f(a) \) is an equivalence in \( Y \).
\end{proof}

\begin{prop} \label{prop:universal_property_weak_composites}
    Let \( X \) be a diagrammatic set.
    The following are equivalent:
    \begin{enumerate}[label=(\alph*)]
	    \item \label{enum:infty_a} \( X \) is an \( (\infty, \infty) \)-category;
	    \item \label{enum:infty_b} for all round diagrams \( u \) of shape \( U \) in \( X \), there exists an equivalence \( c(u) \) of shape \( U \celto \compos{U} \) such that the triangle
        \[
            \begin{tikzcd}
                U & X \\
                {U \celto \compos{U}}
                \arrow["u", from=1-1, to=1-2]
                \arrow[hook', from=1-1, to=2-1]
                \arrow["{c(u)}"', from=2-1, to=1-2]
            \end{tikzcd}
    \]
        commutes;
	\item \label{enum:infty_c} for all round diagrams \( u \) of shape \( U \) in \( X \), there exists a morphism of diagrammatic sets \( \tilde c(u) \colon (U \eqvto \compos{U}) \to X \) such that the triangle
        \[
            \begin{tikzcd}
                U & X \\
                {U \eqvto \compos{U}}
                \arrow["u", from=1-1, to=1-2]
                \arrow[hook', from=1-1, to=2-1]
                \arrow["{\tilde c(u)}"', from=2-1, to=1-2]
            \end{tikzcd}    
    \]
        commutes.
    \end{enumerate}
\end{prop}
\begin{proof}
	Suppose that \( X \) has weak composites and let \( u \) be a round diagram in \( X \).
	Then \( u \) has a weak composite \( \compos{u} \), and by Lemma \ref{lem:equivalences_exhibited_by_cells} the equivalence \( u \simeq \compos{u} \) is exhibited by a cell \( c(u) \) of shape \( U \celto \compos{U} \).
	This proves the implication from \ref{enum:infty_a} to \ref{enum:infty_b}.
	Clearly, \ref{enum:infty_b} also implies \ref{enum:infty_a}.
	Now, suppose condition \ref{enum:infty_b} holds, let \( u \) be a round diagram of shape \( U \), and let \( c(u) \) be the given equivalence in \( X \).
	Letting \( \top \) be the greatest element of \( U \celto \compos{U} \), we have, by definition, \( U \eqvto \compos{U} = \loc{(U \celto \compos{U})}{\set{\top}} \), so by Proposition \ref{prop:universal_property_localisation}, which applies by the first part of the proof, \( c(u) \) extends along the inclusion \( (U \celto \compos{U}) \incl (U \eqvto \compos{U}) \), which proves \ref{enum:infty_c}.
	Finally, assume \ref{enum:infty_c}, let \( u \) be a round diagram in \( X \), and let \( \tilde{c}(u) \) be the given morphism.
	By construction, its value on the cell \( (U \celto \compos{U}) \incl (U \eqvto \compos{U}) \) is an equivalence in \( X \), exhibiting its output boundary as a weak composite of \( u \).
\end{proof}

\begin{prop} \label{prop:inftyn_lifting_property}
	Let \( X \) be an \( (\infty, \infty) \)-category, \( n \in \mathbb{N} \).
	The following are equivalent:
	\begin{enumerate}[label=(\alph*)]
		\item \( X \) is an \( (\infty, n) \)-category;
		\item for all cells \( u \) of shape \( U \) in \( X \), if \( \dim u > n \), then there exists a morphism \( \selfloc u\colon \selfloc U \to X \) such that the triangle
        	\[
            \begin{tikzcd}
                U & X \\
                {\selfloc U}
                \arrow["u", from=1-1, to=1-2]
                \arrow[hook', from=1-1, to=2-1]
                \arrow["\selfloc u"', from=2-1, to=1-2]
            \end{tikzcd}    
    		\]
		commutes.
	\end{enumerate}
\end{prop}
\begin{proof}
	Immediate from the definition and Proposition \ref{prop:universal_property_localisation}.
\end{proof}


\subsection{Marked model structures}

\begin{dfn} [Atomic and molecular horns]
	Let \( U \) be an atom, \( \dim U > 0 \).
	A \emph{molecular horn of \( U \)} is the data of
	\begin{enumerate}
		\item a rewritable submolecule \( V \submol \bd{}{\a}U \), for some \( \a \in \set{+, -} \),
		\item the inclusion \( \lambda_U^V\colon \Lambda_U^V \incl U \), where \( \Lambda_U^V \eqdef \bd{}{}U \setminus \inter{V} \).
	\end{enumerate}
	If \( V = \clset{x} \) is an atom, we write \( \lambda_U^x\colon \Lambda_U^x \incl U \) and call it an \emph{atomic horn}.
\end{dfn}

\begin{rmk}
    In \cite{chanavat2024diagrammatic}, the molecular horns were simply called horns.
\end{rmk}

\begin{dfn} [Marked horns]
	Let \( \lambda_U^x\colon \Lambda_U^x \incl U \) be an atomic horn, let \( \top \) be the greatest element of \( U \), let \( \a \in \set{+, -} \) such that \( x \in \faces{}{\a}U \), and \( k \eqdef \dim{U} - 1 \).
	We say that an inclusion 
	\[
		\lambda^x_U\colon (\Lambda_U^x, A) \incl (U, A')
	\] 
	of marked regular directed complexes is a \emph{marked horn of \( U \)} if there exist molecules \( (\order{i}{L}, \order{i}{R})_{i=1}^k \) such that
	\begin{enumerate}
		\item \( \bd{}{\a}U = \order{k}{L} \cp{k-1} (\ldots \cp{1} (\order{1}{L} \cp{0} \clset{x} \cp{0} \order{1}{R}) \cp{1} \ldots) \cp{k-1} \order{k}{R} \),
		\item \( \dim \order{i}{L}, \dim \order{i}{R} \leq i \) for each \( i \in \set{1, \ldots, k} \),
		\item \( \gr{i}{\order{i}{L}} \cup \gr{i}{\order{i}{R}} \subseteq A \) for each \( i \in \set{1, \ldots, k} \),
	\end{enumerate}
	and, moreover,
	\[
		A' = \begin{cases}
			A \cup \set{x, \top} & \text{if \( \faces{}{-\a}U \subseteq A \),} \\
			A \cup \set{\top} & \text{otherwise.}
		\end{cases}
	\]
\end{dfn}

\begin{comm} \label{comm:marked_horns_are_equations}
	Let \( \lambda^x_U\colon \Lambda_U^x \incl U \) be an atomic horn, let \( \a \in \set{+, -} \) be such that \( x \in \faces{}{\a}U \), let \( k \eqdef \dim{U} - 1 \), and let \( q\colon \Lambda_U^x \to X \) be a morphism of diagrammatic sets.
	If \( k > 0 \), with reference to the terminology of \cite{chanavat2024equivalences},
	\begin{itemize}
		\item \( v \eqdef \restr{q}{\bd{}{-}x} \) and \( w \eqdef \restr{q}{\bd{}{+}x} \) are parallel round diagrams of dimension \( k-1 \) in \( X \), as are \( v' \eqdef \restr{q}{\bd{k-1}{-}U} \) and \( w' \eqdef \restr{q}{\bd{k-1}{+}U} \);
		\item \( b \eqdef \restr{q}{\bd{}{-\a}U} \) is a round diagram of type \( v' \celto w' \) in \( X \), and
		\item \( \restr{q}{\bd{}{\a}U \cap \Lambda_U^x} \) determines a \emph{round context} \( \fun{E}\colon \rd X(v, w) \to \rd X(v', w') \).
	\end{itemize}
	Altogether, these data determine an \emph{equation \( \fun{E}x \qeq b \) in the indeterminate \( x \in \gr{k}{\rd X(v, w)} \)}.
	An extension of \( q \) along \( \lambda^x_U \) is then precisely a pair of a cell \( a\colon v \celto w \) together with a cell \( h\colon \fun{E}a \celto b \) or \( b \celto \fun{E}a \), which is a special case of what we called a \emph{lax} or \emph{colax solution}, respectively.

	Now, suppose that \( \lambda^x_U\colon (\Lambda_U^x, A) \incl (U, A') \) is a marked horn, and suppose \( q \) lifts to a morphism \( q\colon (\Lambda_U^x, A) \to \natmark{X} \) of marked diagrammatic sets.
	Then the conditions on the marking \( A \), together with the assumption that \( q \) sends marked cells to equivalences in \( X \), imply that \( \fun{E} \) is a \emph{weakly invertible context}.
	In this sense, marked horns are classifying objects for equations \( \fun{E}x \qeq b \) with \( \fun{E} \) weakly invertible.
	Moreover, because \( \top \in A' \), an extension of \( q \) along \( \lambda^x_U \) exhibits an equivalence \( \fun{E}a \simeq b \), and if all top\nbd dimensional cells in \( b \) are marked, so that \( b \) is an equivalence, then \( a \) is also an equivalence.

	In this light, the fact that a marked diagrammatic set \( (X, A) \) has the right lifting property against marked horns should be interpreted as the statement that \emph{all marked cells are equivalences in \( X \)}.
\end{comm}

\begin{dfn} [Marked walking equivalence]
    Let \( U \) be an atom of dimension \( > 0 \). 
    The \emph{marked walking equivalence of shape \( U \)} is the marked diagrammatic set
    \[
	    \selflocm{U} \eqdef \left(\selfloc{U}, \dgncell \selfloc{U} \cup \set{U \incl \selfloc{U}}\right).
\]
\end{dfn}

\begin{comm}
	There is an evident entire monomorphism \( \minmark{\selfloc{U}} \incl \selflocm{U} \).
	The fact that a marked diagrammatic set \( (X, A) \) has the right lifting property against these morphisms should be interpreted as the statement that \emph{all equivalences in \( X \) are marked}.
\end{comm}

\begin{dfn} [Walking pair of invertors]
    Let \( U \) be an atom, \( \dim U > 0 \).
    The \emph{walking pair of invertors on \( U \)} is the diagrammatic set \( \walkinv U \) obtained as the colimit in \( \dgmSet \) of the diagram
\[\begin{tikzcd}[column sep=tiny]
	& {\bd{}{+}\lcyl{\bd{}{+}U}{U}} && U && {\bd{}{-}\rcyl{\bd{}{-}U}{U}} \\
	{\bd{}{-}U} && {\lcyl{\bd{}{+}U}{U}} && {\rcyl{\bd{}{-}U}U} && {\bd{}{+}U}
	\arrow[two heads, from=1-2, to=2-1]
	\arrow[hook, from=1-2, to=2-3]
	\arrow[hook', from=1-4, to=2-3]
	\arrow[hook, from=1-4, to=2-5]
	\arrow[hook', from=1-6, to=2-5]
	\arrow[two heads, from=1-6, to=2-7]
\end{tikzcd}\]
    of maps of molecules.
    This is equipped with an evident inclusion \( U \incl \walkinv U \).
\end{dfn}

\begin{comm}
	Let \( u\colon U \to X \) be a cell of type \( v \celto w \) in a diagrammatic set.
	Then, extensions of \( u \) along \( U \incl \walkinv U \) classify pairs of cells \( u^L, u^R\colon w \celto v \) together with pairs of cells \( z \colon u \cp{} u^L \celto \un v \) and \( h \colon \un w \celto u^R \cp{} u \).
\end{comm}

\begin{dfn} [Marked walking pair of invertors]
	Let \( U \) be an atom, \( n \eqdef \dim U > 0 \).
	Then \( \walkinv U \) has exactly
	\begin{enumerate}
		\item two non-degenerate cells of dimension \( n+1 \), classified by the morphisms 
			\(\lcyl{\bd{}{+}U}U \to \walkinv U \) and \( \rcyl{\bd{}{-}U}U \to \walkinv U \) in the colimit cone presenting \( \walkinv U \),
		\item three non-degenerate cells of dimension \( n \), corresponding to the inclusion \( U \incl \walkinv U \) as well as the two morphisms \( \dual{n}{U} \incl \lcyl{\bd{}{+}U}U \to \walkinv U \) and \( \dual{n}{U} \incl \rcyl{\bd{}{-}U}U \to \walkinv U \).
	\end{enumerate}
	We let
	\begin{itemize}
		\item \( \mwalkinv U \) be \( \walkinv U \) with all cells of dimension \( \geq n+1 \) marked,
		\item \( \fmwalkinv U \) be \( \mwalkinv U \) with all cells of dimension \( \geq n \) marked.
	\end{itemize}
\end{dfn}

\begin{comm}
	There is en evident entire monomorphism \( \mwalkinv U \incl \fmwalkinv U \).
	The fact that a marked diagrammatic set \( (X, A) \) has the right lifting property against these morphisms should be interpreted as the statement that \emph{if a cell is weakly invertible up to marked cells, then it is marked}.
	This corresponds to the notion of \emph{saturation} in \cite{loubaton2024inductive}.
\end{comm}

\begin{lem} \label{lem:marked_cellular_model}
    The set of monomorphisms \( M \) defined as
    \[
	    \set{\minmark{\bdmap}_U \colon \minmark{(\bd{}{}U)} \incl \minmark{U} \mid U \in \Ob\atom } \cup \set{t_U \colon \minmark{U} \incl \markmol{U} \mid \text{$U \in \Ob\atom$, $\dim U > 0$}}
\]
    is a cellular model for the monomorphisms of \( \mdgmSet \).
\end{lem}
\begin{proof}
	Let \( i \colon (X, A) \incl (X, B) \) be an entire monomorphism and let \( \kappa \) be an ordinal whose size bounds \( \cell X \).
	Each cell \( a \in B \setminus A \) of shape \( U \) can be marked by means of a pushout along \( t_U \), so by \( \kappa \)\nbd indexed transfinite composition, we construct \( i \) as an element of \( l(r(M)) \).
	Next, let \( i \colon (X, A) \incl (Y, i(A) \cup \dgncell Y) \) be a regular monomorphism.
	By \cite[Proposition 1.17]{chanavat2024diagrammatic}, the set of monomorphisms \( \set{ \bdmap_U \mid U \in \Ob\atom } \) is a cellular model for the monomorphisms of \( \dgmSet \), so \( \minmark{(\fun{U}i)} \) belongs to \( l(r(M)) \).
	Then \( i \) fits in a pushout square
	\[\begin{tikzcd}
	{\minmark{X}} & {(X, A)} \\
	{\minmark{Y}} & {(Y, i(A) \cup \dgncell Y)}
	\arrow[hook, from=1-1, to=1-2]
	\arrow["{\minmark{(\fun{U}i)}}", hook', from=1-1, to=2-1]
	\arrow["i", hook', from=1-2, to=2-2]
	\arrow[hook, from=2-1, to=2-2]
	\arrow["\lrcorner"{anchor=center, pos=0.125, rotate=180}, draw=none, from=2-2, to=1-1]
\end{tikzcd}\]
	where the top morphism is entire, hence belongs to \( l(r(M)) \) by the first part of the proof. 
	We conclude by Remark \ref{rmk:facto_mono_entire_regular}.
\end{proof}

\begin{dfn} [Marked cylinder]
	The \emph{marked cylinder} is the endofunctor \( \markarr \pgray - \) on \( \mdgmSet \), together with the natural transformations \( (\iota^-, \iota^+) \) and \( \sigma \) induced by the morphisms
	\[
		(0^-, 0^+)\colon \pt \amalg \pt \incl \markarr, \quad \quad
		\varepsilon\colon \markarr \surj \pt,
	\]
	of marked regular directed complexes, respectively.
\end{dfn}

\begin{rmk} \label{rmk:marked_cell_cylinder}
    	Given a marked diagrammatic set \( (X, A) \), the set of marked cells of \( \markarr \pgray (X, A) \) is 
    	\[
        \dgncell (\markarr \gray X) \cup \set{0^-} \gray A \cup \set{1} \gray \cell X \cup \set{0^+} \gray A.
	\]
\end{rmk}

\begin{prop}\label{prop:exact_cylinder_mdset}
	The marked cylinder is an exact cylinder on \( \mdgmSet \).
\end{prop}
\begin{proof}
	The axiom \ref{enum:dh0} is satisfied by construction, and \ref{enum:dh1} follows from Lemma \ref{lem:marked_gray_product_preserves_mono} and Lemma \ref{lem:marked_gray_preserve_colimits}.
\end{proof}

\begin{dfn} [Inductive and coinductive model structures] \label{dfn:marked_model_structures}
	Let \( \mdgmSet \) be equipped with
	\begin{enumerate}
		\item the marked cylinder \( ( \markarr \pgray -, (\iota^-, \iota^+), \sigma ) \),
		\item the cellular model \( M \) of Lemma \ref{lem:marked_cellular_model},
	\end{enumerate}
	and let \( n \in \mathbb{N} \cup \set{ \infty } \).
	We define sets of monomorphisms
    \begin{align*}
	    \Jhorn &\eqdef \set{\lambda_U^x \colon (\Lambda_U^x, A) \incl (U, A') \mid \text{$\lambda_U^x$ is a marked horn} }, \\
	    \Jn n &\eqdef \set{\minmark{U} \incl \markmol{U} \mid \text{$U \in \Ob\atom$, $\dim U > n$} }, \\
	    \Jinv &\eqdef \set{\mwalkinv U \incl \fmwalkinv U \mid \text{$U \in \Ob\atom$, $\dim U > 0$}}, \\
	    \Jloc &\eqdef \set{\minmark{\selfloc{U}} \incl \selflocm{U} \mid \text{$U \in \Ob\atom$, $\dim U > 0$}},
    \end{align*}
    and then let
    \[
        \Jind \eqdef \Jhorn \cup \Jn n \cup \Jinv, 
	\quad\quad 
	\Jcoind \eqdef \Jhorn \cup \Jn n \cup \Jloc.
	\]
	By an application of Theorem \ref{thm:olschok_theorem}, the set of monomorphisms
	\begin{itemize}
		\item \( \Jind \) determines the \emph{inductive \( (\infty, n) \)-model structure} on \( \mdgmSet \),
		\item \( \Jcoind \) determines the \emph{coinductive \( (\infty, n) \)-model structure} on \( \mdgmSet \).
	\end{itemize}
\end{dfn}


\subsection{Unmarked model structures}

\begin{dfn} [Reversible cylinder]
	The \emph{reversible cylinder} is the endofunctor
	\[
		\rgray \eqdef \fun{Loc}(\markarr \pgray \minmark{(-)})\colon \dgmSet \to \dgmSet
	\]
	together with the natural transformations \( (\riota^-, \riota^+) \) and \( \rsigma \) obtained by whiskering \( (\iota^-, \iota^+) \) and \( \sigma \) with \( \minmark{(-)} \) on the right and with \( \fun{Loc} \) on the left.
\end{dfn}

\begin{prop} \label{prop:exact_cylinder_dset}
	The reversible cylinder is an exact cylinder on \( \dgmSet \).
\end{prop}
\begin{proof}
	The axiom \ref{enum:dh0} is satisfied by construction.
	Now, \( \minmark{(-)} \) preserves colimits and monomorphisms because it is a full and faithful left adjoint, the marked cylinder is exact by Proposition \ref{prop:exact_cylinder_mdset}, and \( \fun{Loc} \) preserves colimits and monomorphisms by Proposition \ref{prop:functorial_localisation}.
	This proves \ref{enum:dh1}.
\end{proof}

\begin{dfn} [$(\infty, n)$-model structure] \label{dfn:infty_n_model_structure}
	Let \( \dgmSet \) be equipped with
	\begin{enumerate}
		\item the reversible cylinder \( (\rgray, (\riota^-, \riota^+), \rsigma) \),
		\item the cellular model \(\set{\bd{}{} U \incl U \mid U \in \Ob\atom } \) \cite[Remark 2.9]{chanavat2024diagrammatic},
	\end{enumerate}
    	and let \( n \in \mathbb{N} \cup \set{\infty} \). 
	We define sets of monomorphisms
    \begin{align*}
	    \Jcomp &\eqdef \set{ U \incl U \eqvto \compos{U} \mid \text{$U$ is a round molecule} }, \\
            \Jn n &\eqdef \set{U \incl \selfloc{U} \mid \text{$U \in \Ob\atom$, $\dim U > n$}}.
    \end{align*}
    The \emph{\( (\infty, n) \)\nbd model structure} on \( \dgmSet \) is the model structure determined by the set \( \Jcomp \cup \Jn n \) according to Theorem \ref{thm:olschok_theorem}.
\end{dfn}

\noindent We note already that an analogue of the main theorem of \cite{hadzihasanovic2024model} holds in the \( (\infty, n) \)\nbd model structure by purely formal reasons.
Let \( U \) be an atom, and consider the evident surjection \( (U \celto U) \surj U \).
This lifts to a morphism of marked atoms \( \markmol{(U \celto U)} \surj \minmark{U} \), so it extends by Proposition \ref{prop:functorial_localisation} to a surjection of diagrammatic sets \( (U \eqvto U) \surj U \).

\begin{prop} \label{prop:coherent_localisation}
    	Let \( U \) be an atom.
	Then \( (U \eqvto U) \surj U \) is a weak equivalence in the \( (\infty, n) \)\nbd model structure.
\end{prop}
\begin{proof}
	Since \( U \) is an atom, \( \compos{U} = U \), so the inclusion \( \iota\colon U \incl (U \eqvto U) \) is in \( \Jcomp \), and is in particular an acyclic cofibration.
	Now \( \iota \) is a section of \( (U \eqvto U) \surj U \), so we conclude by the 2-out-of-3 property.
\end{proof}

\noindent We readily show that all fibrant objects are \( (\infty, n) \)\nbd categories.

\begin{lem} \label{lem:fibrant_is_inftyn}
	Let \( n \in \mathbb{N} \cup \set{\infty} \) and let \( X \) be fibrant in the \( (\infty, n) \)\nbd model structure.
	Then \( X \) is an \( (\infty, n) \)\nbd category.
\end{lem}
\begin{proof}
	By Proposition \ref{prop:universal_property_weak_composites} combined with Proposition \ref{prop:inftyn_lifting_property}, a diagrammatic set \( X \) has the right lifting property against \( \Jcomp \cup \Jn n \) if and only if it is an \( (\infty, n) \)\nbd category.
\end{proof}

\noindent In Theorem \ref{thm:fibrant_iff_infty_category}, we will also establish the converse.

\section{Characterisation and comparison} \label{sec:comparison}

\subsection{Inductive and coinductive model structures} \label{subsec:ind_coind}

Throughout this section, we fix \( n \in \mathbb{N} \cup \set{\infty} \).

\begin{lem} \label{lem:marked_cell_is_equivalence}
    Let \( (X, A) \) be a marked diagrammatic set and suppose that \( (X, A) \) has the right lifting property against \( \Jhorn \).
    Then \( A \subseteq \eqvcell X \).
\end{lem}
\begin{proof}
	By coinduction, it suffices to show that \( A \subseteq \B(A) \).
	Let \( e\colon V \to X \) be a marked cell and let \( \top \) be the greatest element of \( V \).
	We let
	\[
		U \eqdef \lcyl{\bd{}{+}V}V,
		\quad 
		B \eqdef \set{(0^-, \top)} \cup \faces{}{+}U,
		\quad 
		B' \eqdef B \cup \set{(0^+, \top), (1, \top)}.
	\]
	Then, letting \( x \eqdef (0^+, \top) \), the inclusion \( \lambda^x_U\colon (\Lambda^x_U, B) \incl (U, B') \) is a marked horn, and we have a morphism \( q\colon (\Lambda^x_U, B) \to (X, A) \) restricting to \( \un(\bd{}{-}e) \) on \( \bd{}{+}U \) and to \( e \) on \( V = \clset{(0^-, \top)} \); this is well-defined because all top-dimensional cells in \( \un(\bd{}{-}e) \) are degenerate, so they are marked.
	Extending this morphism along \( \lambda^x_U \) produces marked cells \( e^L \) and \( z\colon e \cp{} e^L \celto \un(\bd{}{-}e) \).
	Dually, we construct a marked horn of \( \rcyl{\bd{}{-}V} \), and the lifting property produces marked cells \( e^R \) and \( h\colon \un(\bd{}{+}e) \celto e^R \cp{} e \).
	This proves \( A \subseteq \B(A) \).
\end{proof}

\begin{lem} \label{lem:fibrant_is_category}
	Let \( (X, A) \) be a marked diagrammatic set and suppose \( (X, A) \) has the right lifting property against \( \Jhorn \cup \Jn n \).
	Then \( X \) is an \( (\infty, n) \)\nbd category.
\end{lem}
\begin{proof}
	Let \( u\colon U \to X \) be a round diagram in \( X \).
	Then \( \minmark{U} \incl \markmol{(U \celto \compos{U})} \) is a marked horn, so by assumption \( u\colon \minmark{U} \to (X, A) \) extends along the horn to \( c(u)\colon \markmol{(U \celto \compos{U})} \to (X, A) \).
	By Lemma \ref{lem:marked_cell_is_equivalence}, \( c(u) \) is an equivalence in \( X \), so by Proposition \ref{prop:universal_property_weak_composites}, \( X \) is an \( (\infty, \infty) \)\nbd category.
	If \( n = \infty \), we are done.
	Otherwise, let \( u\colon U \to X \) be a cell of dimension \( > n \) in \( X \).
	Then, because \( (X, A) \) has the right lifting property against \( \minmark{U} \incl \markmol{U} \), \( u \) is marked.
	By Lemma \ref{lem:marked_cell_is_equivalence}, \( u \) is an equivalence, so \( X \) is an \( (\infty, n) \)\nbd category.
\end{proof}

\begin{lem} \label{lem:equivalences_marked_jloc}
	Let \( (X, A) \) be a marked diagrammatic set and suppose that \( (X, A) \) has the right lifting property against \( \Jhorn \).
	The following are equivalent:
	\begin{enumerate}[label=(\alph*)]
		\item \( (X, A) \) has the right lifting property against \( \Jloc \);
		\item \( A = \eqvcell X \).
	\end{enumerate}
\end{lem}
\begin{proof}
	By Lemma \ref{lem:marked_cell_is_equivalence}, \( A \subseteq \eqvcell X \), and by Lemma \ref{lem:fibrant_is_category}, \( X \) is an \( (\infty, \infty) \)\nbd category.
	Let \( e \) be an equivalence of shape \( U \) in \( X \).
	By Proposition \ref{prop:universal_property_localisation}, \( e \) extends along the inclusion \( U \incl \selfloc{U} \) to a morphism \( \selfloc{U} \to X \), which has a transpose of type \( \minmark{\selfloc{U}} \to (X, A) \).
	This extends to a morphism \( \selflocm{U} \to (X, A) \) if and only if \( e \) is marked, and in either case, \( \eqvcell X \subseteq A \), so \( A = \eqvcell X \).
\end{proof}

\begin{lem} \label{lem:inftyn_natmark_has_rlp}
	Let \( X \) be an \( (\infty, n) \)\nbd category.
	Then \( \natmark{X} \) has the right lifting property against \( \Jcoind \).
\end{lem}
\begin{proof}
	First, consider a marked horn \( \lambda^x_U\colon (\Lambda^x_U, A) \to (U, A') \) and a morphism \( q\colon (\Lambda^x_U, A) \to (U, A') \).
	As detailed in Comment \ref{comm:marked_horns_are_equations}, \( q \) determines an equation \( \fun{E}x \qeq b \), where \( \fun{E} \) is a weakly invertible context; moreover, if \( A' = A \cup \set{\top, x} \), then \( b \) is an equivalence.
	By \cite[Lemma 5.10]{chanavat2024equivalences}, this equation has a weakly unique solution \( a \), given by \( \fun{E}^*b \) where \( \fun{E}^* \) is a weakly invertible weak inverse to \( \fun{E} \).
	Since \( X \) has weak composites, we can assume that \( a \) is a cell, and by Lemma 
	\ref{lem:equivalences_exhibited_by_cells} there exist cells \( h \colon \fun{E}a \eqvto b \) and \( h^* \colon b \eqvto \fun{E}a \) exhibiting the equivalence.
	Moreover, if \( b \) is an equivalence, then \( a \simeq \fun{E}^*b \) is also an equivalence.
	Thus, depending on whether \( x \in \faces{}{-} U \) or \( x \in \faces{}{+} U \), \( h \) or \( h^* \) is an extension of \( q \) along \( \lambda^x_U \).
	This proves the right lifting property of \( \natmark{X} \) against \( \Jhorn \).
	Moreover, by Lemma \ref{lem:equivalences_marked_jloc}, \( \natmark{X} \) has the right lifting property against \( \Jloc \).
	If \( n = \infty \), we are done.
	Otherwise, let \( u \) be a cell of dimension \( > n \) and shape \( U \).
	Then \( u \) is an equivalence in \( X \), hence it is marked.
	It follows that \( u \) extends along \( \minmark{U} \incl \markmol{U} \).
	This proves the right lifting property against \( \Jn n \).
\end{proof}

\begin{lem} \label{lem:pseudo_gray_entire_is_id}
    Let \( m \colon (X, A) \incl (X, B) \) be entire.
    Then \( (\iota^-, \iota^+) \ppnat m \) is the identity on \( \markarr \pgray (X, B) \).
\end{lem}
\begin{proof}
    First of all, \( (\iota^-, \iota^+) \ppnat m \) is entire on \( \arr \gray X \), as its underlying morphism of diagrammatic sets is a pushout-product with an identity morphism.
    Let \( (\arr \gray X, B') \) denote the domain of \( (\iota^-, \iota^+) \ppnat m \). 
    The pushout coprojection \( (X, B) \coprod (X, B) \to (\arr \gray X, B') \) ensures that the cells \( \set{0^\a} \gray B \) belong to \( B' \), for all \( \a \in \set{-, +} \). 
    The other pushout coprojection \( \markarr \pgray (X, A) \to (\arr \gray X, B') \) ensures that the cells \( \set{1} \gray \gr{>0}{\cell X} \) belong to \( B' \). 
    By Remark \ref{rmk:marked_cell_cylinder}, these are all the non-degenerate marked cells in \( \markarr \pgray (X, B) \).
    We conclude that \( (\iota^-, \iota^+) \ppnat m \) is the identity.
\end{proof}

\begin{lem} \label{lem:marked_horns_closed_pushout_product}
	Let \( \lambda^x_U\colon (\Lambda^x_U, A) \incl (U, A') \) be a marked horn.
	Then the pushout-product \( (\iota^-, \iota^+) \ppnat \lambda^x_U \) is a marked horn.
\end{lem}
\begin{proof}
	It follows from \cite[Lemma 3.16]{chanavat2024diagrammatic} that the underlying morphism of \( (\iota^-, \iota^+) \ppnat \lambda^x_U \) is the atomic horn \( \lambda^{(1, x)}_{\arr \gray U}\colon \Lambda^{(1, x)}_{\arr \gray U} \incl \arr \gray U \).
	Assume, without loss of generality, that \( x \in \faces{}{+}U \), the other situation being dual, and let \( (\order{i}{L}, \order{i}{R})_{i=1}^k \) be the molecules in a decomposition of \( \bd{}{+} U \) as by the definition of marked horn.
	Let \( V W \eqdef V \gray W \) for \( V \submol \arr \) and \( W \submol U \).
	By \cite[Proposition 7.2.16]{hadzihasanovic2024combinatorics}, \( \bd{}{-} \arr U \) decomposes into a ``generalised pasting''
	\[
		\set{0^-}U \gencp{k} (\arr\order{k}{R} \gencp{k} (\ldots \gencp{2} (\arr\order{1}{R} \gencp{1} \clset{(1, x)} \gencp{1} \arr\order{1}{L}) \gencp{2} \ldots) \gencp{k} \arr\order{k}{L}),
	\]
	which can be turned into a decomposition of the form
	\[
		\order{k+1}{\overbar{L}} \cp{k} (\ldots \cp{1} (\order{1}{\overbar{L}} \cp{0} \clset{(1, x)} \cp{0} \order{1}{\overbar{R}}) \cp{1} \ldots) \cp{k} \order{k+1}{\overbar{R}}
	\]
	by expanding the generalised pastings as in \cite[Lemma 7.1.4]{hadzihasanovic2024combinatorics}.
	We have
	\begin{itemize}
		\item \( \gr{1}{\order{1}{\overbar{L}}} = \set{0^-} \gray \gr{1}{\order{1}{L}} \) and \( \gr{1}{\order{1}{\overbar{R}}} = \set{0^+} \gray \gr{1}{\order{1}{R}} \),
		\item \( \gr{i}{\order{i}{\overbar{L}}} = \set{0^-} \gray \gr{i}{\order{i}{L}} \cup \set{1} \gray \gr{i-1}{\order{i-1}{R}} \) and \( \gr{i}{\order{i}{\overbar{R}}} = \set{0^+} \gray \gr{i}{\order{i}{R}} \cup \set{1} \gray \gr{i-1}{\order{i-1}{L}} \) for each \( i \in \set{2, \ldots, k} \),
		\item \( \gr{k+1}{\order{k+1}{\overbar{L}}} = \set{0^-} \gray \set{\top_U} \cup \set{1} \gray \gr{k}{\order{k}{R}} \) and \( \gr{k+1}{\order{k+1}{\overbar{R}}} = \set{1} \gray \gr{k}{\order{k}{L}} \),
	\end{itemize}
	all of whose elements are marked in \( \markarr \pgray (U, A') \) by Remark \ref{rmk:marked_cell_cylinder}.
	Moreover, while \( (1, x) \) is also marked regardless of whether \( x \) was, we have
	\[
		\faces{}{+}(\arr \gray U) = \set{0^+} \gray \set{\top} \cup \set{1} \gray \faces{}{-}U,
	\]
	all of whose elements are marked in \( \markarr \pgray (U, A') \).
	We conclude that the pushout-product \( (\iota^-, \iota^+) \ppnat \lambda^x_U \) is of the form
	\[
		\lambda^{(1, x)}_{\arr \gray U}\colon \left(\Lambda^{(1, x)}_{\arr \gray U}, B\right) 
		\incl \left(\arr \gray U, B \cup \set{(1, \top), (1, x)}\right)
	\]
	with \( (1, x) \in \faces{}{-}(\arr \gray U) \) and \( \faces{}{+}(\arr \gray U) \subseteq B \), so it is a marked horn.
\end{proof}

\begin{lem} \label{lem:pushout_product_boundary_is_marked_horn}
	Let \( U \) be an atom, \( \a \in \set{+, -} \), and consider the marked boundary inclusions \( \minmark{\bdmap}_U\colon \minmark{\bd{}{}U} \incl \minmark{U} \) and \( \bdmap_{\markmol{U}}\colon \minmark{\bd{}{}U} \incl \markmol{U} \).
	Then the pushout-products \( \iota^\a \ppnat \minmark{\bdmap}_U \) and \( \iota^\a \ppnat \bdmap_{\markmol{U}} \) are marked horns.
\end{lem}
\begin{proof}
	Suppose without loss of generality that \( \a = + \), the other case being dual.
	Let \( \top \) be the greatest element of \( U \) and \( k \eqdef \dim{U} \).
	It follows from \cite[Lemma 3.16]{chanavat2024diagrammatic} that the underlying morphism of \( \iota^+ \ppnat \minmark{\bdmap}_U \) is the atomic horn \( \lambda^{(0^-, \top)}_U\colon \Lambda^{(0^-, \top)}_{\arr \gray U} \to \arr \gray U \).
	Then \( (0^-, \top) \in \bd{}{-}(\arr \gray U) \), and expanding the equations of \cite[Proposition 7.2.16]{hadzihasanovic2024combinatorics} and letting \( \order{i}{R} \eqdef \arr \gray \bd{i-1}{+}U \) for each \( i \in \set{1, \ldots, k} \), we find that
	\[
		\bd{}{-}(\arr \gray U) = (\ldots(\clset{(0^-, \top)} \cp{0} \order{1}{R})\cp{1} \ldots)\cp{k-1} \order{k}{R}.
	\]
	For each \( i \in \set{1, \ldots, k} \), we have \( \gr{i}{\order{i}{R}} = \set{1} \gray \faces{i-1}{+}U \), all of whose elements are marked both in \( \markarr \pgray \minmark{U} \) and in \( \markarr \pgray \markmol{U} \).
	Since \( (1, \top) \) is the only marked element of \( \markarr \pgray \minmark{U} \) which is not in \( \Lambda^{(0^-, \top)}_{\arr \gray U} \), we conclude that \( \iota^+ \ppnat \minmark{\bdmap}_U \) is a marked horn.
	On the other hand, \( (0^-, \top) \) and \( (0^+, \top) \) are also marked in \( \markarr \pgray \markmol{U} \), which implies that all elements of \( \faces{}{+}(\arr \gray U) \) are marked.
	We conclude that \( \iota^+ \ppnat \bdmap_{\markmol{U}} \) is also a marked horn.
\end{proof}

\begin{lem} \label{lem:fibrant_iff_rlp}
	Let \( (X, A) \) be a marked diagrammatic set and \( J \in \set{\Jind, \Jcoind} \).
	The following are equivalent:
	\begin{enumerate}[label=(\alph*)]
		\item \( (X, A) \) is \( J \)-fibrant;
		\item \( (X, A) \) has the right lifting property against \( J \).
	\end{enumerate}
\end{lem}
\begin{proof}
	One implication is obvious.
	Conversely, suppose that \( (X, A) \) has the right lifting property against \( J \); we need to show that it has the right lifting property against \( \An(J) \).
	Let \( J_M \eqdef \set{ \iota^\a \ppnat m \mid \a \in \set{+, -}, m \in M } \), where \( M \) is the cellular model of Lemma \ref{lem:marked_cellular_model}, and let \( m \in M \).
	Suppose that \( m = \minmark{\bdmap}_U \) for some atom \( U \); by Lemma \ref{lem:pushout_product_boundary_is_marked_horn}, the pushout-product \( \iota^\a \ppnat \minmark{\bdmap}_U \) is in \( \Jhorn \subseteq J \).
	Next, suppose that \( m = t_U \) for some atom \( U \) with greatest element \( \top \).
	Then \( \iota^\a \ppnat t_U \) is an entire monomorphism with codomain \( \markarr \pgray \markmol{U} \).
	The induced monomorphism along the restriction to \( \Lambda^{(0^{-\a}, \top)}_{\arr \gray U} \) is equal to \( \iota^\a \ppnat \bdmap_{\markmol{U}} \), which is a marked horn by Lemma \ref{lem:pushout_product_boundary_is_marked_horn}.
	Since the underlying morphism of \( \iota^\a \ppnat t_U \) is the identity, the right lifting property against this marked horn implies the right lifting property against \( \iota^\a \ppnat t_U \).
	Finally, let \( j \in J \cup J_M \), and consider the morphism \( j' \eqdef (\iota^-, \iota^+) \ppnat j \).
	Then either \( j \in \Jhorn \), in which case \( j' \in \Jhorn \) by Lemma 
	\ref{lem:marked_horns_closed_pushout_product}, or \( j \) is an entire monomorphism, in which case \( j' \) is an identity by Lemma \ref{lem:pseudo_gray_entire_is_id}.
	Thus, closure under \( j \mapsto (\iota^-, \iota^+) \ppnat j \) only saturates \( J \cup J_M \) with identities, and we conclude.
\end{proof}

\begin{thm} \label{thm:coind_fibrant_characterisation}
	Let \( (X, A) \) be a marked diagrammatic set.
	The following are equivalent:
	\begin{enumerate}[label=(\alph*)]
		\item \( (X, A) \) is fibrant in the coinductive \( (\infty, n) \)\nbd model structure;
		\item \( X \) is an \( (\infty, n) \)\nbd category and \( A = \eqvcell X \).
	\end{enumerate}
\end{thm}
\begin{proof}
	One implication is the content of Lemma \ref{lem:fibrant_is_category} and Lemma \ref{lem:equivalences_marked_jloc}.
	Conversely, suppose \( X \) is an \( (\infty, n) \)\nbd category.
	By Lemma \ref{lem:inftyn_natmark_has_rlp}, \( \natmark{X} \) has the right lifting property against \( \Jcoind \).
	We conclude by Lemma \ref{lem:fibrant_iff_rlp}.
\end{proof}

\begin{lem} \label{lem:coind_fibrant_implies_ind_fibrant}
	Let \( X \) be an \( (\infty, n) \)\nbd category.
	Then \( \natmark{X} \) is fibrant in the inductive \( (\infty, n) \)\nbd model structure.
\end{lem}
\begin{proof}
	By Lemma \ref{lem:fibrant_iff_rlp}, we only need to show that \( \natmark{X} \) has the right lifting property against \( \Jinv \).
	Let \( U \) be an atom and consider a morphism \( f\colon \mwalkinv U \to \natmark{X} \).
	This corresponds exactly to the data of a pair of marked cells \( z\colon u \cp{} u^L \celto \un{v} \) and \( h\colon \un{w} \celto u^R \cp{} u \) for some cell \( u\colon v \celto w \) and parallel cells \( u^L, u^R\colon w \celto v \).
	Then \( z \) and \( h \) are equivalences, which by definition means that \( u \) is an equivalence, and by Theorem \ref{thm:properties_of_equivalences}.\ref{enum:wkinv_is_equiv} implies that \( u
^L \) and \( u^R \) are also equivalences.
	We conclude that \( f \) extends along \( \mwalkinv U \incl \fmwalkinv U \).
\end{proof}

\begin{prop} \label{prop:ind_bousfield_coind}
	The coinductive \( (\infty, n) \)\nbd model structure is a left Bousfield localisation of the inductive \( (\infty, n) \)\nbd model structure.
\end{prop}
\begin{proof}
	The coinductive \( (\infty, n) \)\nbd model structure has the same cofibrations as the inductive \( (\infty, n) \)\nbd model structure.
	Moreover, by Theorem \ref{thm:coind_fibrant_characterisation} and Lemma \ref{lem:coind_fibrant_implies_ind_fibrant}, a fibrant object for the coinductive \( (\infty, n) \)\nbd model structure is also fibrant for the inductive \( (\infty, n) \)\nbd model structure.
        Since the two share the same exact cylinder, by the characterisation in Theorem \ref{thm:olschok_theorem}, weak equivalences in the inductive model structure are also weak equivalences in the coinductive.
\end{proof}

\begin{comm}
	With a little effort, it can in fact be shown that the coinductive \( (\infty, n) \)\nbd model structure is precisely the left Bousfield localisation of the inductive \( (\infty, n) \)\nbd model structure at the set \( \Jloc \); we omit the proof to avoid overwhelming this section, as it relies on a number of technical results that will not be used elsewhere in the article.
\end{comm}

\noindent In the case \( n < \infty \), we strengthen this result by showing that the two model structures actually coincide.

\begin{lem} \label{lem:ind_fibrant_iff_coind_fibrant}
	Suppose that \( n < \infty \) and let \( (X, A) \) be a marked diagrammatic set.
	The following are equivalent:
	\begin{enumerate}[label=(\alph*)]
		\item\label{enum:indcoind1} \( (X, A) \) is fibrant in the inductive \( (\infty, n) \)\nbd model structure;
		\item\label{enum:indcoind2} \( X \) is an \( (\infty, n) \)\nbd category and \( A = \eqvcell X \).
	\end{enumerate}
\end{lem}
\begin{proof}
	The implication from \ref{enum:indcoind2} to \ref{enum:indcoind1} is a particular case of Lemma \ref{lem:coind_fibrant_implies_ind_fibrant}.
	Conversely, \( A \subseteq \eqvcell X \) by Lemma \ref{lem:marked_cell_is_equivalence} and \( X \) is an \( (\infty, n) \)\nbd category by Lemma \ref{lem:fibrant_is_category}, so it suffices to show that \( \eqvcell X \subseteq A \).
	Let \( e\colon U \to X \) be an equivalence of type \( v \celto w \) in \( X \).
	If \( k\eqdef \dim{U} > n \), then we know that \( e \) is marked.
	Otherwise, we proceed by downward induction on \( k - (n+1) \geq 0 \), the base case having already been proven.
	By definition, there exist weak inverses \( e^L, e^R\colon w \celto v \) and equivalences \( z\colon e \cp{} e^L \celto \un{v} \), \( h\colon \un{w} \celto e^R \cp{} e \).
	By the inductive hypothesis, \( z \) and \( h \) are marked, so these data uniquely determine a morphism \( \mwalkinv U \to X \).
	Because \( (X, A) \) has the right lifting property against \( \Jinv \), this extends along \( \mwalkinv U \incl \fmwalkinv U \), that is, \( e, e^L, e^R \) are all marked.
\end{proof}

\begin{prop} \label{prop:n_inductive_equal_n_coinductive}
	Suppose that \( n < \infty \).
	Then the inductive and coinductive \( (\infty, n) \)\nbd model structures coincide.
\end{prop}
\begin{proof}
	By Lemma \ref{lem:ind_fibrant_iff_coind_fibrant}, the two model structures have the same fibrant objects, so the proof of Proposition \ref{prop:ind_bousfield_coind} can be dualised.
\end{proof}


\subsection{Marked and unmarked model structures} \label{subsec:marked_unmarked}

Throughout this section, we fix \( n \in \mathbb{N} \cup \set{\infty} \).

\begin{lem} \label{lem:marked_incl_walk_inv_acyclic}
	Let \( U \) be an atom, \( \dim U > 0 \).
    	Then \( \markmol{U} \incl \fmwalkinv U \) is an acyclic cofibration in the inductive and coinductive \( (\infty, n) \)\nbd model structures.
\end{lem}
\begin{proof}
	Let \( \top \) be the greatest element of \( U \), consider the inverted partial Gray cylinders \( \lcyl{\bd{}{+}U}U \) and \( \rcyl{\bd{}{-}U}U \), and let
	\begin{align*}
		A & \eqdef \set{(0^-, \top)} \cup \faces{}{+}\lcyl{\bd{}{+}U}U, 
		  & A' & \eqdef A \cup \set{(0^+, \top), (1, \top)}, \\
		B & \eqdef \set{(0^+, \top)} \cup \faces{}{-}\rcyl{\bd{}{-}U}U,
		  & B' & \eqdef B \cup \set{(0^-, \top), (1, \top)}.
	\end{align*}
	Letting \( \Lambda_L \eqdef \Lambda^{(0^+, \top)}_{\lcyl{\bd{}{+}U}U} \) and \( \Lambda_R \eqdef \Lambda^{(0^-, \top)}_{\rcyl{\bd{}{-}U}U} \), we have that
	\[
		(\Lambda_L, A) \incl (\lcyl{\bd{}{+}U}U, A'), 
		\quad \quad
		(\Lambda_R, B) \incl (\rcyl{\bd{}{-}U}U, B')
	\]
	are both marked horns, hence acyclic cofibrations.
	Then their pushouts 
\[\begin{tikzcd}
	{(\Lambda_L, A)} & {\markmol{U}} \\
	{(\lcyl{\bd{}{+}U}U, A')} & {\markmol{\mathrm{L}} U}
	\arrow[two heads, from=1-1, to=1-2]
	\arrow[hook', from=1-1, to=2-1]
	\arrow[hook', from=1-2, to=2-2]
	\arrow[two heads, from=2-1, to=2-2]
	\arrow["\lrcorner"{anchor=center, pos=0.125, rotate=180}, draw=none, from=2-2, to=1-1]
\end{tikzcd}
	\quad \quad
\begin{tikzcd}
	{(\Lambda_R, B)} & {\markmol{U}} \\
	{(\rcyl{\bd{}{-}U}U, B')} & {\markmol{\mathrm{R}} U}
	\arrow[two heads, from=1-1, to=1-2]
	\arrow[hook', from=1-1, to=2-1]
	\arrow[hook', from=1-2, to=2-2]
	\arrow[two heads, from=2-1, to=2-2]
	\arrow["\lrcorner"{anchor=center, pos=0.125, rotate=180}, draw=none, from=2-2, to=1-1]
\end{tikzcd}\]
	along the surjections that collapse \( \bd{}{+}\lcyl{\bd{}{+}U}U \) and \( \bd{}{-}\rcyl{\bd{}{-}U}U \) onto \(\bd{}{-}U \) and \( \bd{}{+}U \), respectively, are both acyclic cofibrations.
	Since \( \fmwalkinv U \) can be constructed as the pushout of \( \markmol{U} \incl \markmol{\mathrm{L}}U \) and \( \markmol{U} \incl \markmol{\mathrm{R}}U \), we conclude.
\end{proof}

\begin{lem} \label{lem:marked_incl_selflocm_acyclic}
    Let \( U \) be an atom, \( \dim U > 0 \). 
    Then \( \markmol{U} \incl \selflocm{U} \) is an acyclic cofibration in the coinductive \( (\infty, n) \)\nbd model structure.
\end{lem}
\begin{proof}
	Let \( u \) denote the cell \( U \incl \selfloc{U} \) and let \( k \eqdef \dim U \).
	For each \( \ell \geq 0 \), \( s \in \set{L, R}^\ell \), and \( \epsilon \in \set{L, R} \), let \( U_s \eqdef \hcyl{s}U \) and \( U^\epsilon_s \eqdef \dual{k+\ell}{\hcyl{s}U} \).
	There are monomorphisms of diagrammatic sets
	\[
        	\phi_s \colon \selfloc{U}_s \incl \selfloc{U}, \quad\quad \phi_s^\epsilon \colon \selfloc{U}_s^\epsilon \incl \selfloc{U},
	\]
	which, respectively, send \( U_s \incl \selfloc{U}_s \) to \( \hinv{s}u \) and \( U_s^\epsilon \incl \selfloc{U}_s^\epsilon \) to \( \hinv{s}^\epsilon u \).
	Let
	\[
		\selfloc{U}^\ell \eqdef \left(\selfloc{U}, \dgncell \selfloc{U} \cup \bigcup_{i=k}^{k+\ell-1} \gr{i}{(\cell \selfloc{U})} \right),
	\]
	so \( \tilde{U}^0 \equiv \minmark{\tilde{U}} \).
	For each \( s \in \set{L, R}^\ell \) and \( \epsilon \in \set{L, R} \), there is an entire acyclic cofibration with domain \( \selfloc{U}^\ell \) obtained by pushing out the acyclic cofibrations
	\[ 
		\minmark{(\selfloc{U}_s)} \incl \markmol{(\selfloc{U}_s)}, \quad \quad
		\minmark{(\selfloc{U}_s^\epsilon)} \incl \markmol{(\selfloc{U}_s^\epsilon)}
	\]
	respectively along \( j_{\ell - 1} \after \ldots j_{0} \after \minmark{(\phi_s)} \) and along \( j_{\ell - 1} \after \ldots j_{0} \after \minmark{(\phi_s^\epsilon)} \).
	The colimit of this family, which can be obtained by iterative composition of pushouts, is the evident entire monomorphism \( j_\ell\colon \selfloc{U}^\ell \incl \selfloc{U}^{\ell+1} \), which is therefore an acyclic cofibration.
	It follows that the transfinite composite \( j_\infty\colon \selfloc{U} \incl \selfloc{U}^\infty \) of the \( (j_\ell)_{\ell \geq 0} \) is also an acyclic cofibration.
	We have a commutative diagram of cofibrations
    \[
        \begin{tikzcd}
		{\markmol{U}} & {\selflocm{U}} & {\minmark{\selfloc{U}}} \\
            & {\selfloc{U}^\infty}
            \arrow[hook, from=1-1, to=1-2]
            \arrow["i"', hook, from=1-1, to=2-2]
            \arrow["j", hook', from=1-2, to=2-2]
            \arrow["j'"', hook', from=1-3, to=1-2]
            \arrow["{j_\infty}", hook', from=1-3, to=2-2]
        \end{tikzcd}
\]
	where \( j' \in \Jloc \) and \( j_\infty \) are acyclic, so by 2-out-of-3 \( j \) is also acyclic.
	Finally, by construction of the localisation \( U \mapsto \selfloc{U} \), the cofibration \( i \colon \markmol{U} \incl \selfloc{U}^{\infty} \) can be obtained via a transfinite composition of pushouts of the morphisms \( \markmol{(U_s)} \incl \fmwalkinv U_s \) and \( \markmol{(U_s^\epsilon)} \incl \fmwalkinv U_s^\epsilon \), which are all acyclic cofibrations by Lemma \ref{lem:marked_incl_walk_inv_acyclic}.
    	By 2-out-of-3 again, \( \markmol{U} \incl \selflocm{U} \) is an acyclic cofibration.
\end{proof}

\begin{lem} \label{lem:inclusion_into_loc_is_acof}
	Let \( (X, A) \) be a marked diagrammatic set.
	Then the monomorphism \( (X, A) \incl (\loc{X}{A}, A \cup \dgncell\loc{X}{A}) \) is an acyclic cofibration in the coinductive \( (\infty, n) \)\nbd model structure.
\end{lem}
\begin{proof}
	The monomorphism \( (X, A) \incl (\loc{X}{A}, A \cup \dgncell\loc{X}{A}) \) can be constructed as a transfinite composition of pushouts along \( \markmol{U} \incl \selflocm{U} \) indexed by cells \( a\colon U \to X \) in \( \ndcell X \cap A \).
	By Lemma \ref{lem:marked_incl_selflocm_acyclic}, it is an acyclic cofibration.
\end{proof}

\begin{dfn}[Marked reversible cylinder]
	The \emph{marked reversible cylinder} is the functor \( \mrgray\colon \dgmSet \to \mdgmSet \) which sends a diagrammatic set \( X \) to the marked diagrammatic set \( \mrgray X \) which fits in the pushout square
\[\begin{tikzcd}
	{\minmark{(\arr \gray X)}} & {\minmark{(\rgray{X})}} \\
	{\markarr \pgray \minmark{X}} & {\mrgray X}
	\arrow[hook, from=1-1, to=1-2]
	\arrow[hook', from=1-1, to=2-1]
	\arrow["{\psi_X}", hook', from=1-2, to=2-2]
	\arrow["{\varphi_X}", hook, from=2-1, to=2-2]
	\arrow["\lrcorner"{anchor=center, pos=0.125, rotate=180}, draw=none, from=2-2, to=1-1]
\end{tikzcd}\]
	in \( \mdgmSet \), whose vertical morphisms are entire and horizontal morphisms are regular.
	The monomorphisms \( \varphi_X \) and \( \psi_X \) are components of natural transformations, which fit into a commutative diagram
	\[\begin{tikzcd}[column sep=large]
	& {\minmark{(-)} \amalg \minmark{(-)}} \\
	{\markarr \pgray \minmark{(-)}} & \mrgray & {\minmark{(\rgray-)}}
	\arrow["{(\iota^-, \iota^+)}"', hook', from=1-2, to=2-1]
	\arrow["{(\mriota^-, \mriota^+)}", hook', from=1-2, to=2-2]
	\arrow["{(\riota^-, \riota^+)}", hook, from=1-2, to=2-3]
	\arrow["\varphi", hook, from=2-1, to=2-2]
	\arrow["\psi"', hook', from=2-3, to=2-2]
\end{tikzcd}\]
	such that the components of \( (\mriota^-, \mriota^+) \) are induced by \( (\riota^-, \riota^+) \).
\end{dfn}

\begin{lem} \label{lem:marked_rev_cyl_acycof}
	All components of \( \varphi\colon \markarr \pgray \minmark{(-)} \incl \mrgray \) and \( \psi\colon \minmark{(\rgray-)} \incl \mrgray \) are acyclic cofibrations in the coinductive \( (\infty, n) \)\nbd model structure.
\end{lem}
\begin{proof}
	Let \( X \) be a diagrammatic set.
	Then \( \varphi_X \) is a regular monomorphism of the form considered in Lemma \ref{lem:inclusion_into_loc_is_acof}, while the entire monomorphism \( \psi_X \) can be constructed as a transfinite composition of pushouts along \( \minmark{\selfloc{U}} \incl \selflocm{U} \) indexed by cells \( a\colon U \to \arr \gray X \) that are marked in \( \markarr \pgray \minmark{X} \).
\end{proof}

\begin{prop} \label{prop:minmark_is_quillen}
	The adjunction \( \minmark{(-)} \dashv \fun{U} \) is a Quillen adjunction between the \( (\infty, n) \)\nbd model structure and the coinductive \( (\infty, n) \)\nbd model structure.
\end{prop}
\begin{proof}
	We will show that \( \minmark{(-)} \) is a left Quillen functor. 
	It is clear that \( \minmark{(-)} \) preserves monomorphisms, hence cofibrations, so it remains to show that it preserves acyclic cofibrations.
	By \cite[Proposition 2.4.40]{cisinski2019higher}, which applies by Remark \ref{rmk:cisinski_is_olschok}, it suffices to show that \( \minmark{(-)} \) sends all morphisms \( j \in \An(\Jcomp \cup \Jn n) \) to acyclic cofibrations.
	First, we consider the case that \( j \in \Jcomp \), so it is of the form \( U \incl U \eqvto \compos{U} \) for some round molecule \( U \).
	We have a commutative diagram of cofibrations
\[\begin{tikzcd}
	{\minmark{U}} & {\minmark{(U \eqvto \compos{U})}} \\
	{\markmol{(U \celto \compos{U})}} & {\markmol{(U \eqvto \compos{U})}}
	\arrow[hook, from=1-1, to=1-2]
	\arrow[hook', from=1-1, to=2-1]
	\arrow[hook', from=1-2, to=2-2]
	\arrow[hook, from=2-1, to=2-2]
\end{tikzcd}\]
	where the left vertical morphism is in \( \Jhorn \), the right vertical morphism is in \( \Jloc \), and the bottom morphism is an acyclic cofibration by Lemma \ref{lem:marked_incl_selflocm_acyclic}.
	By 2-out-of-3, the top morphism is also an acyclic cofibration.
	If \( n = \infty \), then \( \Jn n \) is empty; otherwise, let \( j \in \Jn n \) be of the form \( U \incl \selfloc{U} \).
	We have a commutative diagram of cofibrations
\[\begin{tikzcd}
	{\minmark{U}} & {\minmark{\selfloc{U}}} \\
	{\markmol{U}} & {\selflocm{U}}
	\arrow[hook, from=1-1, to=1-2]
	\arrow[hook', from=1-1, to=2-1]
	\arrow[hook', from=1-2, to=2-2]
	\arrow[hook, from=2-1, to=2-2]
\end{tikzcd}\]
	where the left vertical morphism is in \( \Jn n \), the right vertical morphism is in \( \Jloc \), and the bottom morphism is an acyclic cofibration by Lemma \ref{lem:marked_incl_selflocm_acyclic}, so the top morphism is also an acyclic cofibration.
	Next, let \( j = \riota^\a \ppnat \bdmap_U \) for some atom \( U \) and \( \a \in \set{+, -} \).
	We have a commutative diagram
\[\begin{tikzcd}[column sep=huge, row sep=large]
	\bullet & \bullet & \bullet \\
	\bullet & \bullet & \bullet
	\arrow["{\varphi_{\bd{}{}U} \cup \idd{U}}", hook, from=1-1, to=1-2]
	\arrow["{\iota^\a \ppnat \minmark{\bdmap}_U}", hook', from=1-1, to=2-1]
	\arrow["{\mriota^\a \ppnat \minmark{\bdmap}_U}", hook', from=1-2, to=2-2]
	\arrow["{\psi_{\bd{}{}U} \cup \idd{U}}"', hook', from=1-3, to=1-2]
	\arrow["{\minmark{(\riota^\a \ppnat \bdmap_U)}}", hook', from=1-3, to=2-3]
	\arrow["{\varphi_U}", hook, from=2-1, to=2-2]
	\arrow["{\psi_U}"', hook', from=2-3, to=2-2]
\end{tikzcd}\]
	where the leftmost vertical morphism is in \( \An(\Jcoind) \), and all the horizontal morphisms are acyclic cofibrations by Lemma \ref{lem:marked_rev_cyl_acycof}.
	By 2-out-of-3, we conclude that all the morphisms in the diagram are acyclic cofibrations.
	Finally, let \( j\colon X \incl Y \) be in \( \An(\Jcomp \cup \Jn n) \), and suppose inductively that \( \minmark{j} \) is an acyclic cofibration.
	Then we have a commutative diagram
\[\begin{tikzcd}[column sep=huge, row sep=large]
	\bullet & \bullet & \bullet \\
	\bullet & \bullet & \bullet
	\arrow["{\varphi_X \cup \idd{Y \amalg Y}}", hook, from=1-1, to=1-2]
	\arrow["{(\iota^-, \iota^+) \ppnat \minmark{j}}", hook', from=1-1, to=2-1]
	\arrow["{(\mriota^-, \mriota^+) \ppnat \minmark{j}}", hook', from=1-2, to=2-2]
	\arrow["{\psi_X \cup \idd{Y \amalg Y}}"', hook', from=1-3, to=1-2]
	\arrow["{\minmark{((\riota^-, \riota^+) \ppnat j)}}", hook', from=1-3, to=2-3]
	\arrow["{\varphi_Y}", hook, from=2-1, to=2-2]
	\arrow["{\psi_Y}"', hook', from=2-3, to=2-2]
\end{tikzcd}\]
	where the leftmost vertical morphism is an acyclic cofibration, and all the horizontal morphisms are acyclic cofibrations by Lemma \ref{lem:marked_rev_cyl_acycof}.
	By 2-out-of-3, we conclude that all the morphisms in the diagram are acyclic cofibrations.
	This completes the proof.
\end{proof}

\begin{thm} \label{thm:fibrant_iff_infty_category}
	Let \( X \) be a diagrammatic set.
	The following are equivalent:
	\begin{enumerate}[label=(\alph*)]
		\item \( X \) is fibrant in the \( (\infty, n) \)\nbd model structure;
		\item \( X \) is an \( (\infty, n) \)\nbd category.
	\end{enumerate}
\end{thm}
\begin{proof}
	One implication is Lemma \ref{lem:fibrant_is_inftyn}.
	For the other implication, suppose that \( X \) is an \( (\infty, n) \)\nbd category.
	By Theorem \ref{thm:coind_fibrant_characterisation}, \( \natmark{X} \) is fibrant in the coinductive \( (\infty, n) \)\nbd model structure, so by Proposition \ref{prop:minmark_is_quillen} \( X = \fun{U}\minmark{X} \) is fibrant in the \( (\infty, n) \)\nbd model structure.
\end{proof}

\begin{lem} \label{lem:forgetful_creates_hpty}
    	Let \( f \), \( g \) be parallel morphisms of marked diagrammatic sets whose codomain is fibrant in the coinductive \( (\infty, n) \)\nbd model structure.
	Then \( f \) is homotopic to \( g \) with the exact cylinder \( \markarr \pgray - \) if and only if \( \fun{U}f \) is homotopic to \( \fun{U}g \) with the exact cylinder \( \rgray \).
\end{lem}
\begin{proof}
	By Theorem \ref{thm:coind_fibrant_characterisation}, the codomain of \( f \) and \( g \) is of the form \( \natmark{Y} \) where \( Y \) is an \( (\infty, n) \)\nbd category.
	Let \( \beta\colon \markarr \pgray (X, A) \to \natmark{Y} \) be a homotopy between \( f \) and \( g \).
	By Proposition \ref{prop:universal_property_localisation}, its underlying morphism extends to a morphism \( \tilde{\beta}\colon \rgray X \to Y \), which is a homotopy between \( \fun{U}f \) and \( \fun{U}g \).
	Conversely, let \( \beta\colon \rgray X \to Y \) be a homotopy between \( \fun{U}f \) and \( \fun{U}g \), and let \( \beta'\colon \minmark{(\rgray X)} \to \natmark{Y} \) be its transpose.
	Because its codomain is fibrant, \( \beta' \) extends along the acyclic cofibration \( \psi_X\colon \minmark{(\rgray X)} \incl \mrgray X \), then restricts along \( \varphi_X \) to \( \markarr \pgray \minmark{X} \).
	The resulting morphism sends cells in \( \iota^-(A) \) and \( \iota^+(A) \) to equivalences, hence extends to a homotopy \( \beta''\colon \markarr \pgray (X, A) \to \natmark{Y} \) between \( f \) and \( g \).
\end{proof}

\begin{thm} \label{thm:marked_quillen_eq_unmarked}
	The adjunction \( \minmark{(-)} \dashv \fun{U} \) is a Quillen equivalence between the \( (\infty, n) \)\nbd model structure and the coinductive \( (\infty, n) \)\nbd model structure.
\end{thm}
\begin{proof}
	By \cite[Corollary 1.3.16.(b)]{hovey2007model}, Theorem \ref{thm:coind_fibrant_characterisation}, and the fact that all objects in \( \dgmSet \) and \( \mdgmSet \) are cofibrant, it suffices to show that
	\begin{enumerate}
		\item for all \( (\infty, n) \)\nbd categories \( X \), the counit \( \minmark{X} \to \natmark{X} \) of the adjunction is a weak equivalence,
		\item \( \minmark{(-)} \) reflects weak equivalences.
	\end{enumerate}
	For the first point, observe that the entire monomorphism \( \minmark{X} \to \natmark{X} \) can be constructed as a transfinite composition of pushouts along \( \minmark{\selfloc{U}} \incl \selflocm{U} \) indexed by non-degenerate equivalences in \( X \), so it is an acyclic cofibration.
	For the second point, by Theorem \ref{thm:fibrant_iff_infty_category}, \( \natmark{(-)} \) determines a bijection between fibrants in the two model structures, and by adjointness, for each diagrammatic set \( X \) and \( (\infty, n) \)\nbd category \( W \), there is a bijection between \( \dgmSet(X, W) \) and \( \mdgmSet(\minmark{X}, \natmark{W}) \).
	Moreover, by Lemma \ref{lem:forgetful_creates_hpty}, the relations \( \approx \) determined on the two sets, respectively, by the reversible and the marked cylinder coincide.
	We conclude by the characterisation of weak equivalences in Theorem \ref{thm:olschok_theorem}.
\end{proof}


\subsection{Weak equivalences between \inftyn-categories} \label{sub:characterisation_ms_diag_sets}

Throughout this section, we fix \( n \in \mathbb{N} \cup \set{\infty} \).
Unless otherwise specified, the terminology we use is relative to the \( (\infty, n) \)\nbd model structure on \( \dgmSet \).

\begin{dfn} [\( \omega \)\nbd equivalence] \label{dfn:omega_equivalence}
    Let \( X \), \( Y \) be \( (\infty, n) \)\nbd categories. 
    A functor \( f\colon X \to Y \) is an \emph{\( \omega \)\nbd equivalence} if
    \begin{enumerate}
	    \item for all \( v \in \gr{0}{\cell Y} \), there exists \( u \in \gr{0}{\cell X} \) such that \( v \simeq f(u) \), and
	    \item for all \( n > 0 \), parallel pairs \( u^-, u^+ \) in \( \gr{n-1}{\rd X} \), and cells \( v \colon f(u^-) \celto f(u^+) \) in \( Y \), there exists a cell \( u \colon u^- \celto u^+ \) in \( X \) such that \( v \simeq f(u) \). 
    \end{enumerate}
\end{dfn}

\noindent We will prove that the weak equivalences between \( (\infty, n) \)\nbd categories are exactly the \( \omega \)\nbd equivalences.
To do so, we adopt the same strategy used for strict \( \omega \)\nbd categories in \cite[Section 20.3]{ara2023polygraphs}, which can be formally translated into the setting of diagrammatic sets, provided we give an appropriate version of the ``transport lemma'' \cite[Lemma 20.3.5]{ara2023polygraphs}, which we do in Lemma \ref{lem:transport_lemma}.
When the proof of a statement is a formal analogue of its strict counterpart, we do not reproduce it but simply indicate a precise reference, trusting the reader to make the appropriate small changes of notations.

\begin{lem} \label{lem:acyclic_fib_is_omega_equivalence}
	Let \( f\colon X \to Y \) be an acyclic fibration of \( (\infty, n) \)\nbd categories.
	Then \( f \) is an \( \omega \)\nbd equivalence.
\end{lem}
\begin{proof}
	The right lifting property of \( f \) against boundary inclusions of atoms implies that the conditions of an \( \omega \)\nbd equivalence are satisfied up to equality.
\end{proof}

\begin{lem} \label{lem:omega_eq_reflect_equivalence}
    Let \( f \colon X \to Y \) be an \( \omega \)\nbd equivalence between \( (\infty, n) \)\nbd categories and let \( u, v \) be parallel cells in \( X \).
    Then \( f(u) \simeq f(v) \) implies \( u \simeq v \).
\end{lem}
\begin{proof}
    Formal analogue of \cite[Proposition 20.1.14]{ara2023polygraphs}.
\end{proof}

\begin{dfn} [Reversible path space]
	By standard facts about locally presentable categories \cite[1.66]{adamek1994locally}, the colimit-preserving endofunctor \( \rgray \) on \( \dgmSet \) has a right adjoint \( \rGamma \), defined on a diagrammatic set \( X \) by
	\[
		\rGamma X\colon U \mapsto \dgmSet(\rgray U, X).
	\]
	This is equipped with natural transformations \( (\rpi^-, \rpi^+) \) and \( \rnu \) obtained as transposes of \( (\riota^-, \riota^+) \) and \( \rsigma \), respectively.
	We call \( \rGamma X \) the \emph{reversible path space of \( X \)}.
	Given a diagram \( \gamma\colon U \to \rGamma X \), we call \( \gamma \) a \emph{reversible cylinder of shape \( U \)}, and write \( \gamma\colon u \cyl v \) for \( u \eqdef \rpi_X^- \after \gamma \) and \( v \eqdef \rpi_X^+ \after \gamma \).
\end{dfn}

\begin{lem} \label{lem:rev_cylinder_is_infty_cat}
    	Let \( X \) be an \( (\infty, n) \)\nbd category.
	Then \( \rGamma X \) is an \( (\infty, n) \)\nbd category.
\end{lem}
\begin{proof}
	It suffices to show that \( \rGamma X \) has the right lifting property against the set \( \Jcomp \cup \Jn n \).
	Transposing along the adjunction \( \rgray \dashv \rGamma \), this is equivalent to showing that \( X \) has the right lifting property against \( \rgray j \) for all \( j \in \Jcomp \cup \Jn n \).
	Since \( \rgray \) is a cylinder object for the \( (\infty, n) \)\nbd model structure and preserves cofibrations, it preserves acyclic cofibrations, so we conclude by fibrancy of \( X \).
\end{proof}

\begin{lem}[Transport lemma] \label{lem:transport_lemma}
    Let \( X \) be an \( (\infty, n) \)\nbd category, \( U \) be an atom, and
    \( \ppair{\gamma^- \colon u^- \cyl v^-}{\gamma^+ \colon u^+ \cyl v^+}\colon \bd{}{} U \to \rGamma X \) a diagram.
    Then
    \begin{enumerate}
        \item for all cells \( u \colon u^- \celto u^+ \), there exists a cell \( v \colon v^- \celto v^+ \) together with a reversible cylinder \( \gamma \colon u \cyl v \) restricting to \( \ppair{\gamma^-}{\gamma^+} \), and
        \item for all cells \( v \colon v^- \celto v^+ \), there exists a cell \( u \colon u^- \celto u^+ \) together with a reversible cylinder \( \gamma \colon u \cyl v \) restricting to \( \ppair{\gamma^-}{\gamma^+} \).
    \end{enumerate}
    Moreover, the two cells \( u, v \) determine each other uniquely up to equivalence.
\end{lem}
\begin{proof}
    The existence part of the statement translates to a right lifting problem for \( \natmark{X} \) against maps in \( \Jhorn \), hence follows from Theorem \ref{thm:coind_fibrant_characterisation}, and the weak uniqueness part follows from \cite[Lemma 5.10]{chanavat2024equivalences}.
\end{proof}

\begin{lem} \label{lem:rev_cylinder_maps_properties}
    	Let \( X \) be an \( (\infty, n) \)\nbd category. 
    	Then
    	\begin{enumerate}
        	\item for all \( \a \in \set{-, +} \), the functor \( \rpi_X^\a \colon \rGamma X \to X \) is an acyclic fibration,
        	\item the functor \( \rnu \colon X \to \rGamma X \) is an \( \omega \)\nbd equivalence.
    	\end{enumerate}
\end{lem}
\begin{proof}
	Formal analogue of \cite[Proposition 20.3.8]{ara2023polygraphs}.
\end{proof}

\begin{dfn}[Reversible mapping path space]
    Let \( f \colon X \to Y \) be a functor between \( (\infty, n) \)\nbd categories.
    The \emph{reversible mapping path space of \( f \)} is the diagrammatic set \( \rGamma_f \) defined by the pullback
    \[
        \begin{tikzcd}
            {\rGamma_f} & {\rGamma Y} \\
            X & Y
            \arrow["q", from=1-1, to=1-2]
            \arrow["p"', from=1-1, to=2-1]
            \arrow["{\rpi^-_Y}", from=1-2, to=2-2]
            \arrow["f"', from=2-1, to=2-2]
    	\arrow["\lrcorner"{anchor=center, pos=0.125}, draw=none, from=1-1, to=2-2]
        \end{tikzcd}
\]
	in \( \dgmSet \).
    Explicitly, cells of shape \( U \) in \( \rGamma_f \) are pairs \( (x, \gamma \colon f(x) \cyl y) \), where \( x \) and \( y \) are cells of shape \( U \) in \( X \) and \( Y \), respectively.
    We let \( p_f \colon \rGamma_f \to Y \) be the functor defined by \( (x, \gamma \colon f(x) \cyl y) \mapsto y \).
\end{dfn}

\begin{lem} \label{lem:omega_eq_iff_rGamma_trivial}
    Let \( f \colon X \to Y \) be a functor of \( (\infty, n) \)\nbd categories.
    Then \( f \) is an \( \omega \)\nbd equivalence if and only if \( p_f \) is an acyclic fibration.
\end{lem}
\begin{proof}
    Formal analogue of \cite[Proposition 20.3.10]{ara2023polygraphs}.
\end{proof}

\begin{prop} \label{prop:2_out_of_3_omega_eq}
    The class of \( \omega \)\nbd equivalences of \( (\infty, n) \)\nbd categories satisfies the 2-out-of-3 property.
\end{prop}
\begin{proof}
    Formal analogue of \cite[Theorem 20.3.11]{ara2023polygraphs}.
\end{proof}

\begin{rmk}
    We note that, by \cite[Lemma 2.4.24]{cisinski2019higher}, the relation of being homotopic is already an equivalence relation between fibrant objects, hence, by Theorem \ref{thm:fibrant_iff_infty_category}, an equivalence relation between \( (\infty, n) \)\nbd categories.
\end{rmk}

\begin{lem} \label{lem:left_right_htpy_inverse_implies_htpy}
	Let \( f \colon X \to Y \) be a functor of \( (\infty, n) \)\nbd categories, and suppose \( f^L, f^R\colon Y \to X \) are functors such that \( f^L \after f \) is homotopic to \( \idd{X} \) and \( f \after f^R \) is homotopic to \( \idd{Y} \).
	Then \( f \) is a homotopy equivalence and both \( f^L \) and \( f^R \) are homotopy inverses of \( f \).
\end{lem}
\begin{proof}
    The statement implies that \( f \) is an isomorphism in the category of fibrant objects up to homotopy, hence it is a homotopy equivalence with the specified homotopy inverses by Theorem \ref{thm:olschok_theorem} and the Yoneda lemma.
\end{proof}

\begin{lem} \label{lem:hpty_eq_preserves_weak_equivalences}
    Let \( f, g : X \to Y \) be homotopic functors of \( (\infty, n) \)\nbd categories.
    Then \( f \) is an \( \omega \)\nbd equivalence if and only if \( g \) is an \( \omega \)\nbd equivalence.
\end{lem}
\begin{proof}
    Let \( \beta \colon X \to \rGamma Y \) be the transpose of a homotopy from \( f \) to \( g \).
    Then \( f = \rpi^- \after \beta \) and \( g = \rpi^+ \after \beta \), so we conclude by 2-out-of-3 for \( \omega \)\nbd equivalences, which applies by Lemma \ref{lem:rev_cylinder_is_infty_cat} and Lemma \ref{lem:rev_cylinder_maps_properties}.
\end{proof}

\begin{thm} \label{thm:w_eq_is_hpty_eq}
    Let \( f \colon X \to Y \) be a functor between \( (\infty, n) \)\nbd categories.
    The following are equivalent:
	\begin{enumerate}[label=(\alph*)]
		\item \( f \) is a weak equivalence;
		\item \( f \) is a homotopy equivalence;
		\item \( f \) is an \( \omega \)\nbd equivalence.
    \end{enumerate}
\end{thm}
\begin{proof}
	The equivalence of the first two follows from Theorem \ref{thm:olschok_theorem} and Theorem \ref{thm:fibrant_iff_infty_category}.
	Suppose that \( f \) is an \( \omega \)\nbd equivalence.
	By Lemma \ref{lem:omega_eq_iff_rGamma_trivial}, \( p_f \colon \rGamma_f \to Y \) is an acyclic fibration.
	Since \( Y \) is cofibrant, \( p_f \) has a section \( s \colon Y \to \rGamma_f \); let \( g \colon Y \to X \) denote the composite of \( s \) with the pullback projection \( p \colon \rGamma_f \to X \), and \( \beta \colon Y \to \rGamma Y \) denote the composite of \( s \) with the other pullback projection \( q \colon \rGamma_f \to \rGamma Y \).
	By construction, for all cells \( v \) in \( Y \),
    	\begin{equation*}
        	s(v) = (g(v), \beta(v) \colon f(g(v)) \cyl v),
    	\end{equation*}
	so the transpose of \( \beta \) is a homotopy from \( f \after g \) to \( \idd{Y} \).
	By Lemma \ref{lem:hpty_eq_preserves_weak_equivalences}, \( f \after g \) is an \( \omega \)\nbd equivalence, so by 2-out-of-3 \( g \) is also an \( \omega \)\nbd equivalence.
	Reproducing the first part of the proof with \( g \) in place of \( f \), we obtain a functor \( f' \colon X \to Y \) such that \( g \after f' \) is homotopic to \( \idd{X} \).
	By Lemma \ref{lem:left_right_htpy_inverse_implies_htpy}, we deduce that \( g \) is a homotopy equivalence with inverse \( f \), so \( f \) is a homotopy equivalence.

	Conversely, suppose that \( f \) is a homotopy equivalence, and let \( g\colon Y \to X \) be a homotopy inverse. 
	Since \( g \after f \) is homotopic to \( \idd{X} \) and \( f \after g \) is homotopic to \( \idd{Y} \), by Lemma \ref{lem:hpty_eq_preserves_weak_equivalences} both \( g \after f \) and \( f \after g \) are \( \omega \)\nbd equivalences.
	We will show that \( f \) is an \( \omega \)\nbd equivalence.
	Let \( v\colon \pt \to Y \), and let \( \beta\colon \rgray Y \to Y \) be a homotopy between \( f \after g \) and \( \idd{Y} \).
	Then \( \beta \after \rgray v \colon \locarr \to Y \) exhibits \( f(g(v)) \simeq v \).
	Next, let \( n > 0 \), let \( u^-, u^+ \in \gr{n-1}{\rd X} \) be parallel, and let \( v\colon f(u^-) \celto f(u^+) \) be a cell in \( Y \).
	Then \( g(v)\colon g(f(u^-)) \celto g(f(u^+)) \), and since \( g \after f \) is an \( \omega \)\nbd equivalence, there exists a cell \( u \colon u^- \celto u^+ \) such that \( g(v) \simeq g(f(u)) \).
	By Theorem \ref{thm:properties_of_equivalences}, \( f(g(v)) \simeq f(g(f(u))) \), and since \( f \after g \) is an \( \omega \)\nbd equivalence, Lemma \ref{lem:omega_eq_reflect_equivalence} implies that \( v \simeq f(u) \).
	This completes the proof.
\end{proof}

\section{Homotopy hypothesis} \label{sec:homotopy}

\noindent Recall from \cite{chanavat2024diagrammatic} that the \emph{Cisinski model structure} on \( \dgmSet \) is the model structure obtained by applying Theorem \ref{thm:olschok_theorem} to the set of molecular horns with the Gray cylinder \( \arr \gray - \) and the cellular model \( \set{\bd{}{} U \incl U \mid U \in \Ob\atom} \).
In this section, we will prove that the Cisinski model structure and the \( (\infty, 0) \)\nbd model structure coincide.
Our proof will go through an intermediate presentation of the same model structure, based on atomic instead of molecular horns.

\begin{dfn}[Atomic Cisinski model structure]
	The \emph{atomic Cisinski model structure} on \( \dgmSet \) is the model structure obtained by applying Theorem \ref{thm:olschok_theorem} to the set \( \Jat \) of atomic horns with the exact cylinder \( \arr \gray - \) and the cellular model \( \set{\bd{}{} U \incl U \mid U \in \atom} \).
\end{dfn}

\begin{lem} \label{lem:atomic_horns_are_saturated}
	The set \( \Jat \) is a generating set of anodyne extensions for the atomic Cisinski model structure.
\end{lem}
\begin{proof}
	The proof of \cite[Proposition 3.18]{chanavat2024diagrammatic}, showing that the set of molecular horns is saturated under \( \An(-) \), restricts to its subset \( \Jat \).
	We conclude by Theorem \ref{thm:olschok_theorem}.
\end{proof}

\noindent The following proof is based on the proof of \cite[Lemma 6.3.2]{henry2018regular}.

\begin{lem} \label{lem:submol_acyclic_atomic_cisinski}
	Let \( U \) be a molecule and \( \iota\colon V \incl U \) be a submolecule inclusion.
	Then, in the atomic Cisinski model structure,
    	\begin{enumerate}
        	\item \( \iota \) is an acyclic cofibration,
		\item if \( U \) is round and \( \iota \) is rewritable, then \( \bd{}{}U \incl U \setminus \inter{V} \) is an acyclic cofibration.
    	\end{enumerate}
\end{lem}
\begin{proof}
	We proceed by induction on \( k \eqdef \dim U \geq 0\); the base case is trivial, so let \( k > 0 \).
    	Because all isomorphisms are acyclic cofibrations, and acyclic cofibrations are closed under composition, for the first statement it suffices to show that the inclusions of the form \( V \incl V \cp{\ell} W \) and \( W \incl V \cp{\ell} W \) are acyclic cofibrations for all \( \ell \geq 0 \) and submolecules \( V, W \submol U \).
	The pasting \( V \cp{\ell} W \) is defined by a pushout of the form
    \[
        \begin{tikzcd}
            {\bd{\ell}{+} V = \bd{\ell}{-} W} & W \\
            V & {V \cp{\ell} W}
            \arrow[hook, from=1-1, to=1-2]
            \arrow[hook', from=1-1, to=2-1]
            \arrow[hook', from=1-2, to=2-2]
            \arrow[hook, from=2-1, to=2-2]
            \arrow["\lrcorner"{anchor=center, pos=0.125, rotate=180}, draw=none, from=2-2, to=1-1]
        \end{tikzcd}
\]
	and since acyclic cofibrations are closed under pushouts, it suffices to prove that \( \bd{\ell}{\a} V \incl V \) is an acyclic cofibration for all \( \ell \geq 0 \), \( \a \in \set{ +, - } \), and \( V \submol U \).
	We proceed by induction on submolecules of \( U \), the base case having already been settled.
	Suppose \( U \) is an atom.
	If \( \ell \geq k \), then the inclusion is an identity.
	Suppose \( \ell = k - 1 \).
	Pick an arbitrary element \( x \in \maxel{\bd{}{-\a}U} \), and construct the atomic horn \( \Lambda^x_U \) as the pushout of cofibrations
\[
	\begin{tikzcd}[column sep=large]
            {\bd{k - 2}{} U} & {\bd{}{-\a} U \setminus \set{x}} \\
            {\bd{}{\a}U} & {\Lambda^x_U}
            \arrow[hook, from=1-1, to=1-2]
            \arrow[hook', from=1-1, to=2-1]
            \arrow[hook', from=1-2, to=2-2]
            \arrow["j", hook, from=2-1, to=2-2]
            \arrow["\lrcorner"{anchor=center, pos=0.125, rotate=180}, draw=none, from=2-2, to=1-1]
        \end{tikzcd}
\]
	in \( \dgmSet \).
	By the inductive hypothesis, since \(\clset{x} \submol \bd{}{-\a}U \) is always rewritable, the top cofibration is acyclic, so \( j \) is acyclic.
	Since the atomic horn \( \lambda^x_U\colon \Lambda^x_U \incl U \) is an anodyne extension, the boundary inclusion \( \bd{}{\a}U \incl U \), which factors as \( \lambda^x_U \after j \), is an acyclic cofibration.
	Finally, if \( \ell < k - 1 \), then by globularity the boundary inclusion factors through \( \bd{\ell}{\a}U \incl \bd{}{\a}U \), and we conclude by the inductive hypothesis.

	Next, suppose that \( U \) splits into proper submolecules \( V \cp{i} W \).
	If \( \ell \leq i \), the the inclusion \( \bd{\ell}{\a}U \) factors through a boundary inclusion into \( V \) or \( W \), and we can apply the inductive hypothesis.
	Suppose \( \ell > i \).
	Then the squares
\[
        \begin{tikzcd}
            {\bd{\ell}{\a} V} & {\bd{\ell}{\a} V \cp i \bd{\ell}{\a} W} && {\bd{\ell}{\a} W} & {V \cp i \bd{\ell}{\a} W} \\
            V & {V \cp i \bd{\ell}{\a} W} && W & {V \cp i W}
            \arrow[hook, from=1-1, to=1-2]
            \arrow[hook', from=1-1, to=2-1]
            \arrow["{j_1}", hook', from=1-2, to=2-2]
            \arrow[hook, from=1-4, to=1-5]
            \arrow[hook', from=1-4, to=2-4]
            \arrow["{j_2}", hook', from=1-5, to=2-5]
            \arrow[hook, from=2-1, to=2-2]
            \arrow[hook, from=2-4, to=2-5]
            \arrow["\lrcorner"{anchor=center, pos=0.125, rotate=180}, draw=none, from=2-2, to=1-1]
            \arrow["\lrcorner"{anchor=center, pos=0.125, rotate=180}, draw=none, from=2-5, to=1-4]
        \end{tikzcd}
\]
	are pushouts in \( \dgmSet \), and by the inductive hypothesis the left vertical morphisms are acyclic cofibrations.
	Then \( j_1, j_2 \) and their composite, which is equal to the inclusion \( \bd{\ell}{\a} U \incl U \), are also acyclic cofibrations.

	Finally, suppose \( U \) is round and \( \iota \) is rewritable.
	As a consequence of \cite[Lemma 4.5.11]{hadzihasanovic2024combinatorics}, there exists a decomposition \( U = L \cp{} V' \cp{} R \) such that 
	\[
		\bd{}{-}V \submol \bd{}{+}L, \quad 
		\bd{}{+}V \submol \bd{}{-}R, \quad
		V' = \bd{}{+}L \subcp{k-1} V = V \cpsub{k-1} \bd{}{-}R,
	\]
	where the first two inclusions are rewritable.
	Moreover, \( \bd{}{\a}V' \) is round and \( L \cap R = \bd{}{\a}V' \setminus \inter{\bd{}{\a}V} \), so in the commutative diagram
\[\begin{tikzcd}
	{\bd{}{-}U = \bd{}{-}L} & {\bd{k-2}{}U} & {\bd{}{+}U = \bd{}{+}R} \\
	L & {L \cap R} & R
	\arrow[hook', from=1-1, to=2-1]
	\arrow[hook', from=1-2, to=1-1]
	\arrow[hook, from=1-2, to=1-3]
	\arrow[hook', from=1-2, to=2-2]
	\arrow[hook', from=1-3, to=2-3]
	\arrow[hook', from=2-2, to=2-1]
	\arrow[hook, from=2-2, to=2-3]
\end{tikzcd}\]
	all the vertical cofibrations are acyclic by the inductive hypothesis.
   	By \cite[Corollary 2.3.17]{cisinski2019higher}, the induced cofibration \( \bd{}{}U \incl L \cup R = U \setminus \inter{V} \) between pushouts of the horizontal cofibrations is also acyclic.
\end{proof}

\begin{lem} \label{lem:molecular_horn_are_acyclic}
	Each molecular horn is an acyclic cofibration in the atomic Cisinski model structure.
\end{lem}
\begin{proof}
	Let \( U \) be an atom and let \( V \submol \bd{}{\a}U \) be a rewritable submolecule.
	The horn \( \Lambda^V_U \) can be constructed as the pushout
\[
        \begin{tikzcd}
		{\bd{}{} U} & {\bd{}{\a} U \setminus \inter V} \\
            	{\bd{}{-\a}U} & {\Lambda^V_U} 
            \arrow[hook, from=1-1, to=1-2]
            \arrow[hook', from=1-1, to=2-1]
            \arrow[hook', from=1-2, to=2-2]
            \arrow["j", hook, from=2-1, to=2-2]
            \arrow["\lrcorner"{anchor=center, pos=0.125, rotate=180}, draw=none, from=2-2, to=1-1]
        \end{tikzcd}
\]
	in \( \dgmSet \), where the top cofibration is acyclic by Lemma \ref{lem:submol_acyclic_atomic_cisinski}.
	Thus the cofibration \( j \) is acyclic.
	Since the boundary inclusion \( \bd{}{-}U \incl U \), which by the same result is an acyclic cofibration, factors as \( \lambda^V_U \after j \), we conclude by 2-out-of-3 that \( \lambda^V_U \) is an acyclic cofibration.
\end{proof}

\begin{prop} \label{prop:atomic_cisinski_equal_molecular_cisinski}
	The atomic Cisinski model structure and the Cisinski model structure coincide.
\end{prop}
\begin{proof}
	The two model structures have the same cofibrations and the same exact cylinder.
	By Theorem \ref{thm:olschok_theorem}, it suffices to show that they have the same fibrant objects.
	Since the set of atomic horns is a subset of the set of molecular horns, it is immediate that fibrant objects in the Cisinski model structure are fibrant in the atomic Cisinski model structure.
	Conversely, by \cite[Theorem 3.22]{chanavat2024diagrammatic} molecular horns are a generating set of acyclic cofibrations in the Cisinski model structure, so we conclude by Lemma \ref{lem:molecular_horn_are_acyclic}.
\end{proof}

\begin{lem} \label{lem:infty_fibrant_iff_cisinski_fibrant}
	Let \( X \) be a diagrammatic set.
	The following are equivalent:
	\begin{enumerate}[label=(\alph*)]
		\item \( X \) is fibrant in the atomic Cisinski model structure;
		\item \( X \) is an \( (\infty, 0) \)\nbd category.
	\end{enumerate}
\end{lem}
\begin{proof}
	Suppose that \( X \) is fibrant in the atomic Cisinski model structure.
	Since the underlying morphism of each marked horn is an atomic horn, it follows that \( \maxmark{X} \) has the right lifting property against \( \Jhorn \).
	Moreover, \( \maxmark{X} \) always has the right lifting property against \( \Jn 0 \).
	It follows from Lemma \ref{lem:fibrant_is_category} that \( X \) is an \( (\infty, 0) \)\nbd category.
	Conversely, suppose that \( X \) is an \( (\infty, 0) \)\nbd category.
	By Theorem \ref{thm:coind_fibrant_characterisation}, \( \natmark{X} \) is fibrant in the coinductive \( (\infty, 0) \)\nbd model structure, and since \( \eqvcell X = \gr{>0}{\cell X} \), we have \( \natmark{X} = \maxmark{X} \).
	Then, transposing the first part of the proof, we see that \( X \) has the right lifting property against \( \Jat \).
	It follows from Lemma \ref{lem:atomic_horns_are_saturated} that \( X \) is fibrant in the atomic Cisinski model structure.
\end{proof}

\begin{thm} \label{thm:cisinski_equal_infty}
    The \( (\infty, 0) \)\nbd model structure and the Cisinski model structure coincide.
\end{thm}
\begin{proof}
	By Proposition \ref{prop:atomic_cisinski_equal_molecular_cisinski}, we may instead compare the atomic Cisinski model structure and the \( (\infty, 0) \)\nbd model structure.
	By definition they have the same cofibrations, and by Lemma \ref{lem:infty_fibrant_iff_cisinski_fibrant}, the same fibrant objects.
	By Theorem \ref{thm:olschok_theorem}, it suffices to show that, for any two parallel morphisms \( f, g \) of diagrammatic sets whose codomain \( Y \) is an \( (\infty, 0) \)\nbd category, \( f \) is homotopic to \( g \) with the exact cylinder \( \rgray \) if and only if \( f \) is homotopic to \( g \) with the exact cylinder \( \arr \gray - \).
	One direction is evident; the other follows from Lemma \ref{lem:forgetful_creates_hpty} and the observation that \( \maxmark{Y} \) is fibrant in the coinductive \( (\infty, 0) \)\nbd model structure.
\end{proof}

\begin{cor}[Homotopy hypothesis] \label{cor:homotopy_hypothesis}
	There is a pair of Quillen equivalences
\[
        \begin{tikzcd}
		\sSet && \dgmSet && \sSet
            \arrow[""{name=0, anchor=center, inner sep=0}, "{i_\Delta}", curve={height=-12pt}, from=1-1, to=1-3]
            \arrow[""{name=1, anchor=center, inner sep=0}, "{(-)_\Delta}", curve={height=-12pt}, from=1-3, to=1-1]
            \arrow[""{name=2, anchor=center, inner sep=0}, "{\Sd_\smallatom}", curve={height=-12pt}, from=1-3, to=1-5]
            \arrow[""{name=3, anchor=center, inner sep=0}, "{\Ex_\smallatom}", curve={height=-12pt}, from=1-5, to=1-3]
            \arrow["\dashv"{anchor=center, rotate=-90}, draw=none, from=0, to=1]
            \arrow["\dashv"{anchor=center, rotate=-90}, draw=none, from=2, to=3]
        \end{tikzcd}
\]
	between the classical model structure on \( \sSet \) and the \( (\infty, 0) \)\nbd model structure on \( \dgmSet \).
\end{cor}
\begin{proof}
	Follows from Theorem \ref{thm:cisinski_equal_infty} combined with \cite[Proposition 3.12, Proposition 3.25]{chanavat2024diagrammatic}.
\end{proof}

\noindent To conclude, we note that our results confirm that the ``bi-invertible'' localisation is a homotopy coherent operation, strengthening the results of \cite{hadzihasanovic2024model} before an explicit comparison with other models of \( (\infty, n) \)\nbd categories for \( n > 0 \).

\begin{prop}
    	Let \( U \) be an atom.
	Then the unique morphism \( \selfloc{U} \surj \pt \) is a weak equivalence in the \( (\infty, 0) \)\nbd model structure.
\end{prop}
\begin{proof}
	The cofibration \( U \incl \selfloc{U} \) is an anodyne extension in the \( (\infty, 0) \)\nbd model structure, while the unique morphism \( U \surj \pt \) is a weak equivalence in the Cisinski model structure by \cite[Lemma 3.8, Theorem 3.10]{chanavat2024diagrammatic}, hence also in the \( (\infty, 0) \)\nbd model structure by Theorem \ref{thm:cisinski_equal_infty}.
	We conclude by 2-out-of-3.
\end{proof}

\section{Outlook and conjectures} \label{sec:further}

\subsection{Equivalence with complicial model structures}

We recall a few basic facts about complicial sets as a model of \( (\infty, n) \)\nbd categories.
A \emph{marked simplicial set}---also known as \emph{stratified simplicial set}---is a pair \( (K, A) \) of a simplicial set \( K \) and a subset \( A \) of \emph{marked simplices} of dimension $> 0$, including all degenerate simplices.
Marked simplicial sets and morphisms that respect the marking form a category \( \msSet \).

Let \( \markmol{\Delta[1]} \) be the 1\nbd simplex with a marking whose only non\nbd degenerate marked simplex is classified by \( \idd{\Delta[1]} \).
In \cite[Definition 1.21]{ozornova2020model}, for each \( n \in \mathbb{N} \cup \set{\infty} \), an \emph{\( n \)\nbd complicial set} is defined as a marked simplicial set which has the right lifting property against a certain set \( \Jn n^\Delta \) of \emph{elementary anodyne extensions}.
In the same article, a model structure is constructed on \( \msSet \) whose fibrant objects are precisely the \( n \)\nbd complicial sets.
While this is obtained with an appeal to \cite[Theorem 100]{verity2008weak}, we believe that it coincides with the one determined by Theorem \ref{thm:olschok_theorem} applied to \( \Jn n^\Delta \) together with
\begin{itemize}
	\item the exact cylinder \( \markmol{\Delta[1]} \times - \) with the natural transformations induced by the coface maps \( (d^0, d^1)\colon \Delta[0] \amalg \Delta[0] \to \Delta[1] \) and codegeneracy map \( s^0\colon \Delta[1] \surj \Delta[0] \), and
	\item the cellular model \( \set{ \bd{}{}\Delta[n] \incl \Delta[n] \mid n \in \mathbb{N} } \).
\end{itemize}
Let \( \simplexcat \) be the simplex category, and recall from \cite[Proposition 9.2.14]{hadzihasanovic2024combinatorics} that we may represent \( \simplexcat \) as the full subcategory of \( \atom \) on the oriented simplices.
The adjunction \( i_\Delta \dashv (-)_\Delta \) of Corollary \ref{cor:homotopy_hypothesis} is obtained by left Kan extension of this representation along the Yoneda embedding.
Now, if \( (K, A) \) is a marked simplicial set, then one checks that the pair \( (i_\Delta K, i_\Delta A \cup \dgncell i_\Delta K ) \) is a marked diagrammatic set, and conversely, if \( (X, A) \) is a marked diagrammatic set, then \( (X_\Delta, A \cap \cell X_\Delta ) \) is a marked simplicial set.
This assignment lifts the adjunction \( i_\Delta \dashv (-)_\Delta \) to an adjunction between \( \msSet \) and \( \mdgmSet \).

We make the following conjecture, which, in the light of \cite{loubaton2023theory}, would imply that our model is equivalent to the models of \( (\infty, n) \)\nbd categories in the geometric cluster. 
\begin{conj} \label{conj:complicial_quillen_eq}
    For all \( n \in \mathbb{N} \cup \set{\infty} \), the adjunction \( i_\Delta \dashv (-)_\Delta \) is a Quillen equivalence between the \( n \)\nbd complicial model structure on \( \msSet \) and the inductive \( (\infty, n) \)\nbd model structure on \( \mdgmSet \).
\end{conj}

\noindent The first step towards proving the conjecture would be to establish that the adjunction is Quillen.
Since \( i_\Delta \) preserves monomorphisms, which are the cofibrations in both model structures, it would suffice to show that it preserves acyclic cofibrations.
By \cite[Proposition 2.4.40]{cisinski2019higher} combined with \cite[Proposition 1.26]{ozornova2020model} showing that \( \Jn n^\Delta \) is a generating set of anodyne extensions, this would follow from sending morphisms in \( \Jn n^\Delta \) to acyclic cofibrations, for which an argument similar to the one used in \cite[Lemma 4.49]{loubaton2024inductive} for marked strict \( \omega \)\nbd categories could be applicable.

Once a Quillen adjunction is established, we would hope to prove that it is a Quillen equivalence along the same lines used to prove Theorem \ref{thm:marked_quillen_eq_unmarked}.
One step towards the proof would be to show that for any two morphisms \( f, g \) with fibrant codomain, \( f \) is homotopic to \( g \) with the exact cylinder \( \markarr \pgray - \) if and only if \( f_\Delta \) is homotopic to \( g_\Delta \) with the exact cylinder \( \markmol{\Delta[1]} \times - \).
We can expect some difficulty in comparing the Gray tensor product \( \otimes \) with the cartesian product \( \times \). 
For the reader who might wonder if the two products are even plausibly homotopic, we refer to \cite[Corollary 3.27]{chanavat2024diagrammatic}, where we showed that this is the case at least when \( n = 0 \).

We note that Corollary \ref{cor:homotopy_hypothesis}, together with Proposition \ref{prop:n_inductive_equal_n_coinductive}, Theorem \ref{thm:marked_quillen_eq_unmarked}, and the Quillen equivalence between the 0\nbd complicial model structure and the classical model structure on simplicial sets imply the conjecture for \( n = 0 \).


\subsection{Comparison with the folk model structure on \omegatit-categories} \label{sec:folk}

We recall from \cite{lafont2010folk} that there is a model structure on the category \( \omegacat \) of strict \( \omega \)\nbd categories and strict functors, called the \emph{folk model structure}, where
\begin{itemize}
	\item cofibrations are cellular extensions, and cofibrant objects are polygraphs,
	\item all objects are fibrant,
	\item weak equivalences are precisely \( \omega \)\nbd equivalences, that is, functors that are essentially surjective on cells of each dimension,
\end{itemize}
where ``essentially'' means ``up to coinductively weakly invertible cells''.
It should be apparent to the reader that this model structure has much in common with our \( (\infty, \infty) \)\nbd model structure.

Let \( \rdcpx \) be the category of regular directed complexes and cartesian maps.
It follows from \cite[Theorem 6.2.35]{hadzihasanovic2024combinatorics} that there is a functor 
\[
	\molecin{-}\colon \rdcpx \to \omegacat 
\]
with the property that, given a regular directed complex \( P \), the strict \( \omega \)\nbd category \( \molecin{P} \) has a minimal set of composition-generators in bijection with the elements of \( P \).
Restricting this functor to \( \atom \), then extending it along the Yoneda embedding produces an adjunction between \( \dgmSet \) and \( \omegacat \), and one could then expect that the adjunction is Quillen between the \( (\infty, \infty) \)\nbd model structure and the folk model structure.
This is \emph{not} the case.
Since all diagrammatic sets are cofibrant, it would be necessary for \( \molecin{U} \) to be a polygraph for all atoms \( U \).
However, we know from \cite[Example 8.2.20]{hadzihasanovic2024combinatorics} that there exists a 4\nbd dimensional molecule \( V \) such that \( \molecin{V} \) is not a polygraph, and this can be embedded into the boundary of a 5\nbd dimensional atom.
We note that 5 is minimal: \( \molecin{U} \) is a polygraph for all atoms \( U \) of dimension \( \leq 4 \).

While there are well-known problems with strict \( \omega \)\nbd categories as models of homotopy types, due mainly to ``strict Eckmann--Hilton'' \cite{simpson1998homotopy, henry2019non}, our results show that the ``extra coherences'' satisfied by pasting of molecules, but not implied by the axioms of strict \( \omega \)\nbd categories in dimension \( \geq 4 \), are actually homotopically sound.
Thus one is led to the conclusion that strict \( 4 \)\nbd categories are, at once, too strict and not strict enough, so they should perhaps be placed at a further ``precategorical'' level, similar to sesquicategories which satisfy associativity but not interchange.
Incidentally, the same issue is going to affect all algebraic models based on a weakening of the algebra of strict \( \omega \)\nbd categories, such as Batanin--Leinster models, and it seems likely to us that none of these models are going to be Quillen equivalent to diagrammatic \( (\infty, n) \)\nbd categories---or, if Conjecture \ref{conj:complicial_quillen_eq} holds, to any of the accepted non-algebraic models of \( (\infty, n) \)\nbd categories---\emph{as long as the analogue of the folk model structure} is used; some further localisation will be necessary.

We take some steps into the description of such a localisation in the strict case.
Recall that, given a regular directed complex \( P \), for all \( x \in P \) we have a unique inclusion \( \mapel{x} \colon \imel{P}{x} \incl P \) with image \( \clset{x} \), and that the canonical map 
\[
        \colim_{x \in P} \imel{P}{x} \to P
\] 
is an isomorphism of regular directed complexes.

\begin{dfn} [Stricter \( \omega \)\nbd category]
	Let \( X \) be a strict \( \omega \)\nbd category.
	We say that \( X \) is a \emph{stricter \( \omega \)\nbd category} if, for all molecules \( U \), the canonical map
	\[
        	\omegacat(\molecin{U}, X) \to \lim_{x \in U} \omegacat(\molecin{\imel{U}{x}}, X)
	\]
	is an isomorphism.
	We denote by \( \omegacat^> \) the full subcategory of \( \omegacat \) on the stricter \( \omega \)\nbd categories.
\end{dfn}
\noindent
Alternatively, the category of stricter \( \omega \)\nbd categories could be defined as the localisation of \( \omegacat \) at the set of morphisms
\[
	\colim_{x \in U} \molecin{\imel{U}{x}} \to \molecin{U}
\] 
indexed by molecules \( U \); we expect that this should exhibit \( \omegacat^> \) as a reflective localisation of \( \omegacat \), and that \( \omegacat^> \) should be locally finitely presentable.
Because stricter \( \omega \)\nbd categories satisfy ``extra equations'', their internal notion of polygraph, that we may call \emph{stricter polygraph}, is actually looser: some objects which are not free with respect to the algebra of strict \( \omega \)\nbd categories will be free with respect to the algebra of stricter \( \omega \)\nbd categories.

By construction, the functor \( \molecin{-} \) has image in \( \omegacat^> \), so we have a ``diagrammatic nerve'' functor \( \fun{N}_\smallatom\colon \omegacat^> \to \dgmSet \) defined by
\[
	X \mapsto \omegacat^>(\molecin{-}, X).
\]
We make the following conjecture.

\begin{conj} \label{conj:folk_quillen_diag}
	There is a ``folk model structure'' on \( \omegacat^> \) where
	\begin{itemize}
		\item cofibrant objects are stricter polygraphs,
		\item all objects are fibrant,
		\item weak equivalences are \( \omega \)\nbd equivalences,
	\end{itemize}
	and such that
	\begin{enumerate}
		\item the subcategory inclusion \( \omegacat^> \incl \omegacat \) is right Quillen with respect to the folk model structure on \( \omegacat \),
		\item the diagrammatic nerve \( \fun{N}_\smallatom\colon \omegacat^> \to \dgmSet \) is right Quillen with respect to the \( (\infty, \infty) \)\nbd model structure on \( \dgmSet \).
	\end{enumerate}
\end{conj}

\noindent Of course, neither adjunction can be a Quillen equivalence.


\subsection{Compatibility with the Gray product}

\noindent We know from \cite{ara2020folkmon} that the folk model structure is monoidal with the Gray product of strict \( \omega \)\nbd categories.
We also already know that the Gray product of diagrammatic sets is biclosed and preserves monomorphisms, hence cofibrations.
Thus, we make the following conjecture.
\begin{conj} \label{conj:gray_mon}
	Let \( n \in \mathbb{N} \cup \set{\infty} \). 
	The \( (\infty, n) \)\nbd model structure on diagrammatic sets is monoidal with the Gray product.
\end{conj}
By Theorem \ref{thm:cisinski_equal_infty} combined with \cite[Theorem 3.23]{chanavat2024diagrammatic}, this is already settled in the case \( n = 0 \).
We expect a proof in the case \( n > 0 \) to closely follow the one for strict \( \omega \)\nbd categories; in particular, after private communication with Dimitri Ara, we believe that a formal analogue of \cite[Lemma 4.2]{ara2020folkmon} may be applicable to diagrammatic sets.

Further, we expect that stricter \( \omega \)\nbd categories also have a Gray product, defined by the same methods used in \cite{ara2020joint}.
We note that the results of \cite[Chapter 11]{hadzihasanovic2024combinatorics} show that there is a class of regular directed complexes, the \emph{acyclic} ones, which is closed under Gray products, and on which \( \molecin{-} \) factors up to natural isomorphism as
\begin{enumerate}
	\item a monoidal functor landing in the category of \emph{strong Steiner complexes},
	\item followed by Steiner's right adjoint \( \nu \) functor.
\end{enumerate}
We can then strengthen Conjecture \ref{conj:folk_quillen_diag} with the expectation that the hypothetical left adjoint from \( \dgmSet \) to \( \omegacat^> \) is a strong monoidal left Quillen functor between monoidal model categories.

\bibliographystyle{alpha}
\small \bibliography{main.bib}

\end{document}